\documentclass[12pt,a4paper]{article}
\usepackage{amsmath}
\usepackage{amsfonts}
\usepackage{amsthm}
\usepackage{amssymb}

\usepackage{enumerate}
\theoremstyle{plain}
\newtheorem{theo}{Theorem}[section]
\newtheorem{lem}[theo]{Lemma}
\newtheorem{prop}[theo]{Proposition}
\theoremstyle{definition}
\newtheorem{definition}[theo]{Definition}
\theoremstyle{remark}
\newtheorem{rem}[theo]{Remark}
\numberwithin{equation}{section}

\newcommand{\C}{\mathbb{C}}

\newcommand{\R}{\mathbb{R}}
\newcommand{\N}{\mathbb{N}}
\newcommand{\M}{\mathbb{M}}
\newcommand{\divrg}{\textrm{div}\,}

\title{Recent results about the
detection of unknown boundaries and inclusions in elastic plates
\thanks{Work supported by PRIN No. 20089PWTPS}}
\author{Antonino Morassi\thanks{Dipartimento di Ingegneria Civile e Architettura,
Universit\`a degli Studi di Udine, via Cotonificio 114, 33100
Udine, Italy. E-mail: \textsf{antonino.morassi@uniud.it}}, \  Edi
Rosset\thanks{Dipartimento di Matematica e Geoscienze,
Universit\`a degli Studi di Trieste, via Valerio 12/1, 34127
Trieste, Italy. E-mail: \textsf{rossedi@univ.trieste.it}} \ and
Sergio Vessella\thanks{Dipartimento di Matematica per le
Decisioni, Universit\`a degli Studi di Firenze, Via delle Pandette
9, 50127 Firenze, Italy. E-mail:
\textsf{sergio.vessella@dmd.unifi.it}}}

\date{}

\begin{document}
\maketitle

\noindent \textbf{Abstract.} In this paper we review some recent
results concerning inverse problems for thin elastic plates. The
plate is assumed to be made by non-homogeneous linearly elastic
material belonging to a general class of anisotropy. A first group
of results concerns uniqueness and stability for the determination
of unknown boundaries, including the cases of cavities and rigid
inclusions. In the second group of results, we consider upper and
lower estimates of the area of unknown inclusions given in terms
of the work exerted by a couple field applied at the boundary of
the plate. In particular, we extend previous size estimates for
elastic inclusions to the case of cavities and rigid inclusions.

\medskip

\medskip


\medskip

\medskip


\section{Introduction} \label{sec:intro}

The problems we consider in the present paper belong to a more
general issue that has evolved in the last fifteen years in the
field of inverse problems. Such an issue collects the problems of
determining, by a finite number of boundary measurements, unknown
boundaries and inclusions entering the boundary value problems for
partial differential equations and systems of elliptic and
parabolic type. Such problems arise in nondestructive techniques
by electrostatic measurements \cite{Ka-Sa}, \cite{In}, in thermal
imaging \cite{Vo-Mo}, \cite{BrC96}, in elasticity theory
\cite{Nakamura}, \cite{Bonnet-Constantinescu}, \cite{M-R2}, and in
many other similar applications \cite{Is2}. In this paper we try
to enlighten the different facets of the issue fixing our
attention on the theory of thin elastic plates. In Section
\ref{sec:model} we give a self contained derivation of the
Kirchhoff-Love plate model on which such a theory is based.

We begin with the problem of the determination of a rigid inclusion
embedded in a thin elastic plate.

Let $\Omega$ denote the middle plane of the plate. We assume that
$\Omega$ is a bounded domain of $\R^2$ of class $C^{1,1}$. Let $h$
be its constant thickness, $h<<\hbox{diam}(\Omega)$. The rigid
inclusion $D$ is modeled as an open simply connected domain
compactly contained in $\Omega$, with boundary of class $C^{1,1}$.
The transversal displacement $w \in H^2(\Omega)$ of the plate
satisfies the following mixed boundary value problem, see, for
example, \cite{l:fi} and \cite{l:gu},
\begin{center}
\( {\displaystyle \left\{
\begin{array}{lr}
     {\rm div}({\rm div} (
      {\mathbb P}\nabla^2 w))=0,
      & \mathrm{in}\ \Omega \setminus \overline {D},
        \vspace{0.25em}\\
      ({\mathbb P} \nabla^2 w)n\cdot n=-\widehat M_n, & \mathrm{on}\ \partial \Omega,
          \vspace{0.25em}\\
      {\rm div}({\mathbb P} \nabla^2 w)\cdot n+(({\mathbb P} \nabla^2
      w)n\cdot \tau),_s
      =(\widehat M_\tau),_s, & \mathrm{on}\ \partial \Omega,
        \vspace{0.25em}\\
      w|_{\overline{D}} \in \mathcal{A}, &\mathrm{in}\ \overline{D},
          \vspace{0.25em}\\
        w^e_{,n} = w^i_{,n}, &\mathrm{on}\ \partial {D},
          \vspace{0.25em}\\
\end{array}
\right. } \) \vskip -8.9em
\begin{eqnarray}
& & \label{eq:1.dir-pbm-incl-rig-1}\\
& & \label{eq:1.dir-pbm-incl-rig-2}\\
& & \label{eq:1.dir-pbm-incl-rig-3}\\
& & \label{eq:1.dir-pbm-incl-rig-4}\\
& & \label{eq:1.dir-pbm-incl-rig-5}
\end{eqnarray}

\end{center}
coupled with the \emph{equilibrium conditions} for the rigid
inclusion $D$
\begin{multline}
  \label{eq:1.equil-rigid-incl}
    \int_{\partial D} \left ( {\rm div}({\mathbb P} \nabla^2 w)\cdot n+(({\mathbb P} \nabla^2
  w)n\cdot \tau),_s \right )g - (({\mathbb P} \nabla^2 w)n\cdot n)
  g_{,n} =0, \\   \quad \hbox{for every } g\in \mathcal{A},
\end{multline}
where $\mathcal{A}$ denotes the space of affine functions. In the
above equations, $n$ and $\tau$ are the unit outer normal and the
unit tangent vector to $\Omega \setminus \overline{D}$,
respectively, and we have defined $w^e \equiv w|_{\Omega \setminus
\overline{D}}$ and $w^i \equiv w|_{\overline{D}}$. Moreover,
$\widehat M_{\tau}$, $\widehat M_n$ are the twisting and bending
components of the assigned couple field $\widehat M$,
respectively. The plate tensor $\mathbb P$ is given by $\mathbb P
= \frac{h^3}{12}\mathbb C$, where $\mathbb C$ is the elasticity
tensor describing the response of the material of the plate.
We assume that $\mathbb C$ has cartesian components
$C_{ijkl}$, $i,j,k,l=1,2$, which satisfy the standard symmetry
conditions \eqref{eq:sym-conditions-C-components}, the regularity assumption \eqref{eq:3.bound} and the
strong convexity condition \eqref{eq:3.convex}.

Given any $\widehat{M}\in H^{-\frac{1}{2}}(\partial\Omega, \R^2)$,
satisfying the compatibility conditions $\int_{\partial \Omega}
\widehat{M}_i = 0$, for $i=1,2$, problem
\eqref{eq:1.dir-pbm-incl-rig-1}--\eqref{eq:1.equil-rigid-incl}
admits a solution $w \in H^{2}(\Omega)$, which is uniquely
determined up to addition of an affine function.

Let us denote by $\Gamma$ an open portion
within $\partial \Omega$ representing the part of the boundary
where measurements are taken.

The inverse problem consists in determining $D$ {}from the
measurement of $w$ and $w_{,n}$ on $\Gamma$. For instance, the
uniqueness issue can be formulated as follows:
\textit{Given two solutions $w_i$ to
\eqref{eq:1.dir-pbm-incl-rig-1}--\eqref{eq:1.equil-rigid-incl} for
$D=D_i$, $i=1,2$, satisfying
\begin{equation}
  \label{eq:1.cond_Dir1}
  w_1 = w_2, \hbox{ on } \Gamma,
\end{equation}
\begin{equation}
  \label{eq:1.cond_Dir2}
  w_{1,n}=w_{2,n}, \hbox{ on } \Gamma,
\end{equation}
does $D_1=D_2$ hold?}

It is convenient to replace each solution $w_i$ introduced
above with $v_i=w_i-g_i$, where $g_i$ is the affine function which
coincides with $w_i$ on $\partial D_i$, $i=1,2$. By this approach,
maintaining the same letter to denote the solution, we rephrase
the equilibrium problem
\eqref{eq:1.dir-pbm-incl-rig-1}--\eqref{eq:1.dir-pbm-incl-rig-5}
in terms of the following mixed boundary value problem with
homogeneous Dirichlet conditions on the boundary of the rigid
inclusion

\begin{center}
\( {\displaystyle \left\{
\begin{array}{lr}
     {\rm div}({\rm div} (
      {\mathbb P}\nabla^2 w))=0,
      & \mathrm{in}\ \Omega \setminus \overline {D},
        \vspace{0.25em}\\
      ({\mathbb P} \nabla^2 w)n\cdot n=-\widehat M_n, & \mathrm{on}\ \partial \Omega,
          \vspace{0.25em}\\
      {\rm div}({\mathbb P} \nabla^2 w)\cdot n+(({\mathbb P} \nabla^2
      w)n\cdot \tau),_s
      =(\widehat M_\tau),_s, & \mathrm{on}\ \partial \Omega,
        \vspace{0.25em}\\
      w=0, &\mathrm{on}\ \partial D,
          \vspace{0.25em}\\
        \frac{\partial w}{\partial n} = 0, &\mathrm{on}\ \partial {D},
          \vspace{0.25em}\\
\end{array}
\right. } \) \vskip -8.9em
\begin{eqnarray}
& & \label{eq:1.dir-pbm-incl-rig-1bis}\\
& & \label{eq:1.dir-pbm-incl-rig-2bis}\\
& & \label{eq:1.dir-pbm-incl-rig-3bis}\\
& & \label{eq:1.dir-pbm-incl-rig-4bis}\\
& & \label{eq:1-dir-pbm-incl-rig-5bis}
\end{eqnarray}

\end{center}
\noindent coupled with the \emph{equilibrium conditions}
\eqref{eq:1.equil-rigid-incl}, which has  a unique solution $w\in
H^2(\Omega\setminus \overline{D})$. Therefore, the uniqueness
question may be rephrased as follows:

\textit{Given two solutions $w_i$ to
\eqref{eq:1.dir-pbm-incl-rig-1bis}--\eqref{eq:1-dir-pbm-incl-rig-5bis}, \eqref{eq:1.equil-rigid-incl} for $D=D_i$, $i=1,2$,
satisfying, for some $g\in \mathcal{A}$
\begin{equation}
  \label{eq:cond_Dir_stab}
  w_1 - w_2 = g,\quad (w_{1} - w_{2})_{,n} =  g_{,n}, \quad \hbox{on }\Gamma,
\end{equation}
does $D_1=D_2$ hold?}

Obviously, the inverse problem above is equivalent to the
determination of the portion $\partial D$ of the boundary of
$\Omega\setminus\bar D$ in the boundary value problem
\eqref{eq:1.dir-pbm-incl-rig-1bis}-\eqref{eq:1-dir-pbm-incl-rig-5bis}.

In the stability issue, instead of \eqref{eq:1.cond_Dir1} and
\eqref{eq:1.cond_Dir2}, we assume
 \begin{equation}
    \label{eq:new_10}
\min_{g \in \cal{A}} \left\{\|w_1 - w_2 -g \|_{L^2(\Gamma)}+
\left\|(w_1 - w_2 -g)_{,n} \right\|_{L^2(\Gamma)}\right\}\leq
\epsilon,
\end{equation}
for some $\epsilon >0$, and we ask for the following estimate
\begin{equation}
    \label{eq:new_20}
d_{\cal H}(\partial D_1,\partial D_2) \leq \eta(\epsilon),
\end{equation}
where $\eta(\epsilon)$ is a suitable infinitesimal function.

The uniqueness for the problem above has been proved in \cite{M-R2}
under the a priori assumption of $C^{3,1}$ regularity of $\partial
D$ and with only one nontrivial couple field. Here, by nontrivial we
mean that
\begin{equation}
    \label{eq:new_25}
(\widehat M_n,\widehat M_{\tau,s}) \not\equiv 0.
\end{equation}

Concerning the stability issue, in \cite{M-R-V5} we have proved a
log-log type estimate, namely in inequality \eqref{eq:new_20} we
have $\eta(\epsilon)=O\left((\log|\log
{\epsilon}|)^{-\alpha}\right)$, where the positive parameter
$\alpha$ depends on the a priori data, see Theorem \ref{theo:Main}
below for a precise statement.

In \cite{M-R-V2} the inverse problem of determining a cavity in an
elastic plate has been faced. We recall that in such a case
conditions (1.4)-(1.5) are replaced by homogeneous Neumann
boundary conditions, which are much more difficult to handle with
respect to Dirichlet boundary conditions arising in the case of
rigid inclusions. For this reason a uniqueness result has been
established making \emph{two} linearly independent boundary
measurements. In \cite{M-R2} it has also been proved a uniqueness
result for a variant of the problems considered above, that is the
case of a plate whose boundary has an unknown and inaccessible
portion where $\widehat M=0$. In this case, thanks to the more
favorable geometric situation, one measurement suffices to detect
the unknown boundary portion. The corresponding stability results
for these two cases have not yet been proved.

The methods used to prove the above mentioned uniqueness and
stability results are based on unique continuation properties and
quantitative estimates of unique continuation for solutions to the
plate equation \eqref{eq:1.dir-pbm-incl-rig-1}. Since such
properties and estimates are consequences of the three sphere
inequality for solutions to equation
\eqref{eq:1.dir-pbm-incl-rig-1}, we will discuss a while about the
main features of such inequality.

The three sphere inequality for solutions to partial differential
equations and systems has a long and interesting history that
intertwines with the issue of unique continuation properties and
the issue of stability estimates \cite{Al-M}, \cite{Ho85},
\cite{Is}, \cite{Jo}, \cite{Lan}, \cite{Lav}, \cite{LRS},
\cite{L-N-W}, \cite{L-Nak-W}. In many important cases, the three
sphere inequality is the elementary tool to prove various types of
quantitative estimates of unique continuation such as, for
example, stability estimates for the Cauchy problem, smallness
propagation estimates and quantitative evaluation of the vanishing
rate of solutions to PDEs. Such questions have been intensively
studied in the context of second order equations of elliptic and
parabolic type. We refer to \cite{A-R-R-V} and \cite{Ve2} where
these topics are widely investigated for such types of equations.

The three sphere inequality for equation \eqref{eq:1.dir-pbm-incl-rig-1} has been proved in \cite{M-R-V5} under
the very general assumption that the elastic material of the plate
is anisotropic and obeys the so called \textit{dichotomy condition}.
Roughly speaking, such a condition implies that the plate operator
at the left hand side of \eqref{eq:1.dir-pbm-incl-rig-1} can be
written as $L_2L_1+Q$, where $L_2$, $L_1$ are second order elliptic
operators with $C^{1,1}$ coefficients and $Q$ is a third order
operator with bounded coefficients. For more details we refer to
\eqref{3.D(x)bound}-\eqref{3.D(x)bound 2} below and \cite{M-R-V5}. A
simplified version of such inequality is the following one
\begin{equation}
  \label{eq:3sph}
   \int_{B_{r_{2}}(x_0)}|\nabla ^2w|^{2} \leq C
   \left(  \int_{B_{r_{1}}(x_0)}|\nabla ^2w|^{2}
   \right)^{\delta}\left(  \int_{B_{r_{3}}(x_0)}|\nabla ^2w|^{2}
   \right)
   ^{1-\delta},
\end{equation}
for every $r_1<r_2<r_3$, where $\delta\in (0,1)$ and $C$ depend only on the parameters
related to the regularity, ellipticity and dichotomy conditions assumed on $\mathbb C$, and on the ratios $r_1/r_2$, $r_2/r_3$; in particular,  $\delta$ and $C$ \textit{do not depend} on $w$.

Previously, the three sphere inequality was proved in \cite{M-R-V1},
for the isotropic plate (that is
 $C_{ijkl}(x)=\delta_{ij}\delta_{kl}\lambda(x)+\left(\delta_{ik}\delta_{jl}
    +\delta_{il}\delta_{jk}\right)\mu(x), \quad \hbox{}{i,j,k,l=1,2}$) and in \cite{Ge} for
    the class of fourth (and higher) order elliptic equation $\mathcal{L}u=0$
    where $\mathcal{L}=L_2L_1$ and $L_2$, $L_1$ are second order elliptic equation with $C^{1,1}$ coefficients.

The proof of the stability result, of which we give a sketch in Subsection \ref{subsec:stab},
has essentially the same structure of the proofs of analogous stability results in the following context:

a) Second order elliptic equations: \cite{AlR},
\cite{BeVe}, \cite{Ro}(two variables elliptic equations);
\cite{A-B-R-V},\cite{AlBRVe2}, \cite{Si} (several variables elliptic
equations)

b) Second order parabolic equation: \cite{CRoVe1}, \cite{CRoVe2},
\cite{DcRVe}, \cite{B-Dc-Si-Ve}, \cite{Ve1}, \cite{Ve2}

c) Elliptic systems: \cite{M-R2}, \cite{M-R3} (elasticity);
\cite{Ba} (Stokes fluid)

It is important to say that the stability estimates proved in the papers of list a) and b) are of
logarithmic type, that is an optimal rate of convergence, as shown by counterexamples (\cite{Al1},
\cite{DcR} for case a) and \cite{DcRVe}, \cite{Ve2} for case b)). The
stability estimates proved in the papers of list c) and in the case
of plate equation are of log-log type, that is with a worse rate of convergence.
It seems difficult to improve such an estimate. We believe that the main difficulty to get such an
improvement is due to the lack of quantitative estimates of
\textit{strong unique continuation property at the boundary}.
In order to give an idea of the crucial point which marks the difference between
these cases, let us notice that, by iterated application of the three sphere
inequality \eqref{eq:3sph} it can be proved there exists
$\bar\rho>0$ such that for every $\rho\in (0,\bar\rho)$ and every
$\bar x\in \partial {D_j}$, $j=1,2$, the following inequality holds
true
\begin{equation}
   \label{eq:new 60}
\int_{B_\rho(\bar x)\cap(\Omega\setminus \overline{D_j})}|\nabla^2
w_j|^2\geq C\exp\left(-A{\rho}^{-B}\right) ,
\end{equation}
where $A>0$, $B>0$ and $C>0$ only depend on the a priori
information, in particular they depend by the quantity (frequency)
   \begin{equation}
   \frac{\|\widehat{M}\|_{L^2(\partial
   \Omega ,\R^2)}}{ \|\widehat{M}\|_{H^{-\frac{1}{2}}(\partial \Omega,\R^2)}}.
   \end{equation}
In cases a) and b), it is possible to prove a refined form of
inequality \eqref{eq:new 60}, in which the exponential
term is replaced with a positive power of $\rho$, obtaining a quantitative estimate of
strong unique continuation property at the boundary.
It has been shown in \cite{A-B-R-V}
that this is a key ingredient in proving that the stability estimate
for the corresponding inverse problem with unknown boundaries in the
conductivity context is not worse than logarithm. This mathematical
tool is available for second order elliptic, \cite{AE}, and parabolic equations, \cite{EsFeVe},
but is not currently available for elliptic systems and plate
equation. This happens even in the simplest case of isotropic
material, and this is the reason for the presence of a double
logarithm in our stability estimate.
Finally, as remarked in \cite{M-R-V4}, it seems hopeless the
possibility that solutions to \eqref{eq:1.dir-pbm-incl-rig-1} can
satisfy even a strong unique continuation property in the interior,
without any a priori assumption on the anisotropy of the material,
see also \cite{Ali}. Regarding this point, our dichotomy condition
\eqref{3.D(x)bound}-\eqref{3.D(x)bound 2} basically contains the
same assumptions under which the unique continuation property holds
for a fourth order elliptic equation in two variables.

In the present paper we have also proved constructive upper and
lower estimates of the area of a rigid inclusion or of a cavity,
$D$, in terms of an easily expressed quantity related to work.
More precisely, suppose we make the following diagnostic test. We
take a reference plate, i.e. a plate without inclusion or cavity,
and we deform it by applying a couple field $\widehat M$ at the
boundary $\partial\Omega$. Let $W_0$ be the work exerted in
deforming the specimen. Next, we repeat the same experiment on a
possibly defective plate. The exerted work generally changes and
assumes, say, the value $W$. We are interested in finding
constructive estimates, from above and from below, of the
\textit{area} of $D$ in terms of the difference $|W-W_0|$. In
order to prove such estimates we proceed along the path outlined
in \cite{M-R-V1} and \cite{M-R-V6} in which the inclusion inside
the plate is made by different elastic material. In this
introduction we illustrate such \emph{intermediate} case, since
the scheme of the mathematical procedure is fairly simple to
describe. With regard to this intermediate case we also want to
stress that, in contrast to the extreme cases, there are not
available any kind of uniqueness result for the inverse problem of
determining inclusion $D$ from the knowledge of a finite number of
measurements on the boundary . This appears to be an extremely
difficult problem. In fact, despite the wide research developed in
this field, a general uniqueness result has not been obtained yet
even in the simpler context arising in electrical impedance
tomography (which involves a second order elliptic equation), see,
for instance, \cite{Is} and \cite{l:aless99} for an extensive
reference list.

Denoting, as above, by $w$ the transversal displacement of the
plate and by $\widehat M_{\tau}$, $\widehat M_n$ the twisting and
bending components of the assigned couple field $\widehat M$,
respectively, the infinitesimal deformation of the defective plate
is governed by the fourth order Neumann boundary value problem
\begin{center}
\( {\displaystyle \left\{
\begin{array}{lr}
     \divrg(\divrg ((\chi_{\Omega \setminus {D}}{\mathbb P} + \chi_{D}
  \widetilde {\mathbb P})\nabla^2 w))=0,
  & \mathrm{in}\ \Omega,
        \vspace{0.25em}\\
     ({\mathbb P} \nabla^2 w)n\cdot n=-\widehat {M}_n, & \mathrm{on}\ \partial
  \Omega,
          \vspace{0.25em}\\
      \divrg({\mathbb P} \nabla^2 w)\cdot n+(({\mathbb P} \nabla^2
  w)n\cdot \tau),_s
  =(\widehat M_\tau),_s, & \mathrm{on}\ \partial
  \Omega.
        \vspace{0.25em}\\
\end{array}
\right. } \) \vskip -6.2em
\begin{eqnarray}
& & \label{eq:intr.equation_with_D}\\
& & \label{eq:intr.bc1_with_D}\\
& & \label{eq:intr.bc2_with_D}
\end{eqnarray}

\end{center}
In the above equations, $\chi_D$ denotes the characteristic
function of $D$.  The plate tensors $\mathbb P$, $\widetilde
{\mathbb P}$ are given by
\begin{equation}
    \label{eq:PandC}
    \mathbb P = \frac{h^3}{12}\mathbb C, \quad \widetilde {\mathbb P} =
    \frac{h^3}{12} \widetilde {\mathbb C},
\end{equation}
where $\mathbb C$ is the elasticity tensor describing the response
of the material in the reference plate  $\Omega$ and satisfies the usual symmetry
conditions \eqref{eq:sym-conditions-C-components}, regularity condition \eqref{eq:3.bound},
strong convexity condition \eqref{eq:3.convex} and the \emph{dichotomy condition},
whereas
$\widetilde {\mathbb C}$ denotes the (unknown) corresponding
tensor for the inclusion $D$.

The work exerted by the couple field
$\widehat M$ has the expression
\begin{equation}
  \label{eq:intr.W}
  W=-\int_{\partial\Omega}\widehat M_{\tau,s}w+\widehat M_nw,_n.
  \end{equation}
When the inclusion $D$ is absent, the equilibrium problem
\eqref{eq:intr.equation_with_D}-\eqref{eq:intr.bc2_with_D} becomes

\begin{center}
\( {\displaystyle \left\{
\begin{array}{lr}
     \divrg(\divrg ( {\mathbb
P}\nabla^2 w_0))=0,
  & \mathrm{in}\ \Omega,
        \vspace{0.25em}\\
     ({\mathbb P} \nabla^2 w_0)n\cdot n=-\widehat M_n, & \mathrm{on}\ \partial
  \Omega,
          \vspace{0.25em}\\
      \divrg({\mathbb P} \nabla^2 w_0)\cdot n+(({\mathbb P} \nabla^2
  w_0)n\cdot \tau),_s
  =(\widehat M_\tau),_s, & \mathrm{on}\ \partial
  \Omega,
        \vspace{0.25em}\\
\end{array}
\right. } \) \vskip -6.2em
\begin{eqnarray}
& & \label{eq:4.equation_without D}\\
& & \label{eq:4.bc1_without_D}\\
& & \label{eq:4.bc2_without_D}
\end{eqnarray}

\end{center}
where $w_0$ is the transversal displacement of the reference
plate. The corresponding work exerted by $\widehat M$ is given by
\begin{equation}
  \label{eq:intr.W_0}
  W_0=-\int_{\partial\Omega}\widehat M_{\tau,s}w_0+\widehat M_nw_{0,n}.
  \end{equation}

In  \cite{M-R-V6} the following result has been proved. Assuming that the following \textit{fatness-condition} is satisfied
\begin{equation}
    \label{eq:intr.fatness}
    \mathrm{area}
    \left (
    \{x \in D | \ dist\{ x, \partial D\} > h_1 \}
    \right )
    \geq \frac{1}{2} \mathrm{area} (D),
\end{equation}
where $h_1$ is a given positive number, then
\begin{equation}
    \label{eq:intr.stime}
    C_1 \left |
    \frac{W-W_0}{W_0}
    \right |
    \leq
    \mathrm{area}(D)
    \leq
    C_2 \left |
    \frac{W-W_0}{W_0}
    \right |,
\end{equation}
where the constants $C_1$, $C_2$ only depend on the a priori data.
Besides the assumptions on the plate
tensor $\mathbb C$ given above, estimates \eqref{eq:intr.stime} are established under some
suitable assumption on the jump $\widetilde {\mathbb C} -
\mathbb C$.

In the \emph{extreme} cases corresponding to a rigid inclusion or
a cavity $D$, the estimates are a little more involute than
\eqref{eq:intr.stime} and additional regularity conditions on the
boundary $\partial D$ and on the plate tensor $\mathbb P$ are
generally  required. The main difference between \emph{extreme}
and \emph{intermediate} cases lies in the estimate from below of
$|D|$. Indeed, in the former case we use regularity estimates (in
the interior) for the reference solution $w_0$ to equation
\eqref{eq:4.equation_without D}, whereas in the latter we combine
regularity estimates, trace and Poincar\'{e} inequalities. On the
other hand, the argument used for the estimate from above of $|D|$
is essentially the same as in \cite{M-R-V6} and involves
quantitative estimates of unique continuation in the form of three
sphere inequality for the hessian $\nabla^2 w_0$. It is exactly at
this point that the dichotomy condition
\eqref{3.D(x)bound}--\eqref{3.D(x)bound 2} on the tensor $\mathbb
C$ is needed.

The analogous bounds in plate theory for \emph{intermediate} inclusions
were first obtained when the reference plate satisfies isotropic conditions, \cite{M-R-V1},
extended to anisotropic materials satisfying the dichotomy conditions, \cite{M-R-V6},
obtained in a weaker form for general inclusions in absence of the fatness condition, \cite{M-R-V3},
and recently in the context of shallow shells in
\cite{l:dlw} and \cite{l:dlvw}. The reader is referred to \cite{l:kss}, \cite{l:ar},
\cite{l:ars}, \cite{BeFrVe} for size estimates of inclusions in the context of
the electrical impedance tomography and to \cite{ikehata98},
\cite{l:amr03}, \cite{l:amr04}, \cite{l:amrv08} for corresponding
problems in two and three-dimensional linear elasticity. See also
\cite{l:LDiN09} for an application of the size estimates approach
in thermography.

Size estimates for \emph{extreme} inclusions were obtained in \cite{l:amr02} for electric conductors and in \cite{M-R0}
for elastic bodies, see also \cite{l:amr03}.

The paper is organized as follows. In Section \ref{sec:notation}
we collect some notation. In Section \ref{sec:model} we provide a
derivation of the Kirchhoff-Love model of the plate. In Section
\ref{sec:uniq_stab} we present the uniqueness and stability
results concerning the determination of rigid inclusions, and the
uniqueness results for the case of cavities and unknown boundary
portions. In particular we have focused our attention on the case
of rigid inclusions and, for a better comprehension of the
arguments, we have recalled the proof of the uniqueness result,
using it as a base for a sketch of the more complex proof of the
stability result. Section \ref{sec:size-estimates-extreme}
contains the estimates of the area of \emph{extreme} inclusions.

\section{Notation} \label{sec:notation}

Let $P=(x_1(P), x_2(P))$ be a point of $\R^2$.
We shall denote by $B_r(P)$ the disk in $\R^2$ of radius $r$ and
center $P$ and by $R_{a,b}(P)$ the rectangle
$R_{a,b}(P)=\{x=(x_1,x_2)\ |\ |x_1-x_1(P)|<a,\ |x_2-x_2(P)|<b \}$. To simplify the notation,
we shall denote $B_r=B_r(O)$, $R_{a,b}=R_{a,b}(O)$.

\begin{definition}
  \label{def:2.1} (${C}^{k,1}$ regularity)
Let $\Omega$ be a bounded domain in ${\R}^{2}$. Given $k\in\N$, we say that a portion $S$ of
$\partial \Omega$ is of \textit{class ${C}^{k,1}$ with
constants $\rho_{0}$, $M_{0}>0$}, if, for any $P \in S$, there
exists a rigid transformation of coordinates under which we have
$P=0$ and
\begin{equation*}
  \Omega \cap R_{\frac{\rho_0}{M_0},\rho_0}=\{x=(x_1,x_2) \in R_{\frac{\rho_0}{M_0},\rho_0}\quad | \quad
x_{2}>\psi(x_1)
  \},
\end{equation*}
where $\psi$ is a ${C}^{k,1}$ function on
$\left(-\frac{\rho_0}{M_0},\frac{\rho_0}{M_0}\right)$ satisfying
\begin{equation*}
\psi(0)=0, \quad \psi' (0)=0, \quad \hbox {when } k \geq 1,
\end{equation*}
\begin{equation*}
\|\psi\|_{{C}^{k,1}\left(-\frac{\rho_0}{M_0},\frac{\rho_0}{M_0}\right)} \leq M_{0}\rho_{0}.
\end{equation*}

\medskip
\noindent When $k=0$ we also say that $S$ is of
\textit{Lipschitz class with constants $\rho_{0}$, $M_{0}$}.
\end{definition}
\begin{rem}
  \label{rem:2.1}
  We use the convention to normalize all norms in such a way that their
  terms are dimensionally homogeneous with the $L^\infty$ norm and coincide with the
  standard definition when the dimensional parameter equals one, see \cite{M-R-V5} for details.

\end{rem}

Given a bounded domain $\Omega$ in $\R^2$ such that $\partial
\Omega$ is of class $C^{k,1}$, with $k\geq 1$, we consider as
positive the orientation of the boundary induced by the outer unit
normal $n$ in the following sense. Given a point
$P\in\partial\Omega$, let us denote by $\tau=\tau(P)$ the unit
tangent at the boundary in $P$ obtained by applying to $n$ a
counterclockwise rotation of angle $\frac{\pi}{2}$, that is
$\tau=e_3 \times n$,
where $\times$ denotes the vector product in $\R^3$, $\{e_1,
e_2\}$ is the canonical basis in $\R^2$ and $e_3=e_1 \times e_2$.
Given any connected component $\cal C$ of $\partial \Omega$ and
fixed a point $P\in\cal C$, let us define as positive the
orientation of $\cal C$ associated to an arclength
parametrization $\varphi(s)=(x_1(s), x_2(s))$, $s \in [0, l(\cal
C)]$, such that $\varphi(0)=P$ and
$\varphi'(s)=\tau(\varphi(s))$, where $l(\cal C)$ denotes the
length of $\cal C$.

Throughout the paper, we denote by $\partial_i u$, $\partial_s u$, and $\partial_n u$
the derivatives of a function $u$ with respect to the $x_i$
variable, to the arclength $s$ and to the normal direction $n$,
respectively, and similarly for higher order derivatives.

We denote by $\mathbb{M}^2$ the space of $2 \times 2$ real valued
matrices and by ${\mathcal L} (X, Y)$ the space of bounded linear
operators between Banach spaces $X$ and $Y$.

For every $2 \times 2$ matrices $A$, $B$ and for every $\mathbb{L}
\in{\mathcal L} ({\mathbb{M}}^{2}, {\mathbb{M}}^{2})$, we use the
following notation:
\begin{equation}
  \label{eq:2.notation_1}
  ({\mathbb{L}}A)_{ij} = L_{ijkl}A_{kl},
\end{equation}
\begin{equation}
  \label{eq:2.notation_2}
  A \cdot B = A_{ij}B_{ij}, \quad |A|= (A \cdot A)^{\frac {1} {2}}.
\end{equation}
Notice that here and in the sequel summation over repeated indexes
is implied.

Finally, let us introduce the linear space of the affine functions
on $\R^2$
\begin{equation*}
  \label{eq:2.notation_3ter}
  \mathcal{A}=\{g(x_1,x_2)=ax_1+bx_2+c,\  a,b,c \in\R\}.
\end{equation*}

\section{The Kirchhoff-Love plate model } \label{sec:model}

In the last two decades different methods were used to provide new
justification of the theory of thin plates. Among these, we recall
the method of asymptotic expansion \cite{Ciarlet-Destuynder}, the
method of internal constraints \cite{Podio-Guidugli},
\cite{Lembo-Podio}, the theory of $\Gamma$-convergence in
conjunction with appropriate averages \cite{Anzellotti} or on
rescaled domain and with rescaled displacements \cite{Bourquin},
\cite{Paroni-1}, and weak convergence methods on a rescaled domain
and with rescaled displacements \cite{Ciarlet}. We refer the
interested reader to \cite{Paroni-2} for a recent account of the
advanced results on this topic. The present section has a more
modest aim: to show how to deduce the equations governing the
statical equilibrium of an elastic thin plate following the
classical approach of the Theory of Structures.

Let us consider a thin plate $\Omega \times \left [ - \frac{h}{2},
\frac{h}{2} \right ]$ with middle surface represented by a bounded
domain $\Omega$ in $\R^2$ having uniform thickness $h$, $h <<
$diam$(\Omega)$, and boundary $\partial \Omega$ of class
$C^{1,1}$. Only in this section, we adopt the convention that
Greek indexes assume the values $1,2$, whereas Latin indexes run
{}from $1$ to $3$.

We follow the direct approach to define the infinitesimal
deformation of the plate. In particular, we restrict ourselves to
the case in which the points $x=(x_1,x_2)$ of the middle surface
$\Omega$ are subject to transversal displacement $w(x_1,x_2)e_3$,
and any transversal material fiber $\{x\}\times \left [ -
\frac{h}{2}, \frac{h}{2} \right ]$, $x\in \Omega$, undergoes an
infinitesimal rigid rotation $\omega(x)$, with $\omega(x)\cdot e_3
=0$. In this section we shall be concerned exclusively with
regular functions on their domain of definition. For example, the
above functions $w$ and $\omega$ are such that $w \in C^\infty(
\overline{\Omega}, \R)$ and $\omega \in C^\infty(
\overline{\Omega}, \R^3)$. These conditions are unnecessarily
restrictive, but this choice simplifies the mechanical formulation
of the equilibrium problem. The above kinematical assumptions
imply that the displacement field present in the plate is given by
the following three-dimensional vector field:
\begin{equation}
  \label{eq:anto-1.2}
  u(x,x_3)=w(x)e_3 + x_3 \varphi (x), \quad x\in
  \overline{\Omega}, \ |x_3| \leq \frac{h}{2},
\end{equation}
where
\begin{equation}
  \label{eq:anto-1.3}
  \varphi (x) = \omega (x) \times e_3, \quad x\in
  \overline{\Omega}.
\end{equation}
By \eqref{eq:anto-1.2} and \eqref{eq:anto-1.3}, the associated
infinitesimal strain tensor $E[u]\in \M^3$ takes the form
\begin{equation}
  \label{eq:anto-2.3}
  E[u](x,x_3) \equiv (\nabla u)^{sym}(x,x_3)= x_3 (\nabla_x
  \varphi(x)  )^{sym} + (\gamma(x) \otimes e_3)^{sym},
\end{equation}
where $\nabla_x (\cdot)=  \frac{\partial}{\partial x_\alpha}
(\cdot ) e_\alpha $ is the surface gradient operator,
$\nabla^{sym}(\cdot)= \frac{1}{2} ( \nabla (\cdot) + \nabla^T
(\cdot))$, and
\begin{equation}
  \label{eq:anto-2.4}
  \gamma(x)=\varphi(x) + \nabla_x w(x).
\end{equation}
Within the approximation of the theory of infinitesimal
deformations, $\gamma$ is the angular deviation between the
transversal material fiber at $x$ and the normal direction to the
deformed middle surface of the plate at $x$. In Kirchhoff-Love
theory it is assumed that every transversal material fiber remains
normal to the deformed middle surface, e.g. $\gamma =0$ in
$\Omega$.

The traditional deduction of the mechanical model of a thin plate
follows essentially {}from integration over the thickness of the
corresponding three-dimensional quantities. In particular, taking
advantage of the infinitesimal deformation assumption, we can
refer the independent variables to the initial undeformed
configuration of the plate.

Let us introduce an arbitrary portion $\Omega' \times \left [ -
\frac{h}{2}, \frac{h}{2} \right ]$ of plate, where $\Omega'
\subset \subset \Omega$ is a subdomain of $\Omega$ with regular
boundary. Consider the material fiber $\{x\}\times \left [ -
\frac{h}{2}, \frac{h}{2} \right ]$ for $x\in \partial \Omega'$ and
denote by $t(x,x_3,e_\alpha)\in \R^3$, $|x_3| \leq \frac{h}{2}$,
the \textit{traction vector} acting on a plane containing the
direction of the fiber and orthogonal to the direction $e_\alpha$.
By Cauchy's Lemma \cite{Truesdell} we have
$t(x,x_3,e_\alpha)=T(x,x_3)e_\alpha$, where $T(x,x_3) \in \M^{3}$
is the (symmetric) Cauchy stress tensor at the point $(x,x_3)$.
Denote by $n$ the unit outer normal vector to $\partial \Omega'$
such that $n\cdot e_3=0$. To simplify the notation, it is
convenient to consider $n$ as a two-dimensional vector belonging
to the plane $x_3=0$ containing the middle surface $\Omega$ of the
plate. By the classical Stress Principle for plates
\cite{Villaggio}, we postulate that the two complementary parts
$\Omega'$ and $\Omega \setminus \Omega'$ interact with one another
through a filed of force vectors $R=R(x,n)\in \R^3$ and couple
vectors $M=M(x,n) \in \R^3$ assigned per unit length at $x \in
\partial \Omega'$. Denoting by
\begin{equation}
  \label{eq:anto-4.1}
  R(x,e_\alpha) = \int_{-h/2}^{h/2} t(x,x_3, e_\alpha) dx_3
\end{equation}
the force vector (per unit length) acting on a direction
orthogonal to $e_\alpha$ and passing through $x \in
\partial \Omega'$, the contact force $R(x,n)$ can be expressed as
\begin{equation}
  \label{eq:anto-4.2}
  R(x,n) = T^\Omega (x) n, \quad x \in \partial \Omega',
\end{equation}
where the \textit{surface force tensor} $T^\Omega (x) \in
\M^{3\times 2}$ is given by
\begin{equation}
  \label{eq:anto-4.3}
  T^\Omega (x) = R(x,e_\alpha) \otimes e_\alpha, \quad \hbox{in }
  \Omega.
\end{equation}
Let $P=I-e_3 \otimes e_3$ be the projection of $\R^3$ along the
direction $e_3$. $T^{\Omega}$ is decomposed additively by $P$ in
its \textit{membranal} and \textit{shearing} component
\begin{equation}
  \label{eq:anto-4.4}
  T^\Omega  = PT^\Omega + (I-P)T^\Omega \equiv T^{\Omega(m)}+
  T^{\Omega(s)},
\end{equation}
where, following the standard nomenclature in plate theory, the
components $T_{\alpha \beta}^{\Omega(m)}$ ($=T_{\beta
\alpha}^{\Omega(m)}$), $\alpha, \beta =1,2$, are called the
\textit{membrane forces} and the components $T_{3
\beta}^{\Omega(s)}$, $\beta=1,2$, are the \textit{shear forces}
(also denoted as $T_{3 \beta}^{\Omega(s)} = Q_\beta$). The
assumption of infinitesimal deformations and the hypothesis of
vanishing in-plane displacements of the middle surface of the
plate allow us to take
\begin{equation}
  \label{eq:anto-5.1}
  T^{\Omega(m)} = 0, \quad \hbox{in } \Omega.
\end{equation}
Denote by
\begin{equation}
  \label{eq:anto-5.2}
  M(x,e_\alpha) = \int_{-h/2}^{h/2} x_3 e_3 \times t (x,x_3,e_\alpha)
  dx_3, \quad \alpha=1,2,
\end{equation}
the contact couple acting at $x \in \partial \Omega'$ on a
direction orthogonal to $e_\alpha$ passing through $x$. Note that
$M(x,e_\alpha) \cdot e_3=0$ by definition, that is $M(x,e_\alpha)$
actually is a two-dimensional couple field belonging to the middle
plane of the plate. Analogously to \eqref{eq:anto-4.2}, we have
\begin{equation}
  \label{eq:anto-5.3}
  M(x,n) = M^\Omega (x) n, \quad x \in \partial \Omega',
\end{equation}
where the \textit{surface couple tensor} $M^\Omega (x) \in
\M^{2\times 2}$ has the expression
\begin{equation}
  \label{eq:anto-4.3}
  M^\Omega (x) = M(x,e_\alpha) \otimes e_\alpha.
\end{equation}
A direct calculation shows that
\begin{equation}
  \label{eq:anto-6.1}
  M(x,e_\alpha) = e_3 \times e_\beta M_{\beta \alpha}(x),
\end{equation}
where
\begin{equation}
  \label{eq:anto-6.2}
  M_{\beta \alpha}(x)=\int_{-h/2}^{h/2}x_3 T_{\beta
  \alpha}(x,x_3)dx_3, \quad \alpha, \beta =1,2,
\end{equation}
are the \textit{bending moments} (for $\alpha=\beta$) and the
\textit{twisting moments} (for $\alpha \neq \beta$) of the plate
at $x$ (per unit length).

The differential equilibrium equation for the plate follows {}from
the integral mechanical balance equations applied to any subdomain
$\Omega'  \subset \subset \Omega$ \cite{Truesdell}. Denote by
$q(x)e_3$ the external transversal force per unit area acting in
$\Omega$. The statical equilibrium of the plate is satisfied if
and only if the following two equations are simultaneously
satisfied:
\begin{center}
\( {\displaystyle \left\{
\begin{array}{lr}
    \int_{\partial \Omega'} T^{\Omega}n ds + \int_{\Omega'}qe_3
    dx=0,
        \vspace{0.25em}\\
    \int_{\partial \Omega'} \left ( (x-x_0) \times T^{\Omega}n
    + M^{\Omega}n \right ) ds + \int_{\Omega'} (x-x_0)\times qe_3 dx=0,
          \vspace{0.25em}\\
\end{array}
\right. } \) \vskip -4.5em
\begin{eqnarray}
& & \label{eq:anto-7.1}\\
& & \label{eq:anto-7.2}
\end{eqnarray}
\end{center}
for every subdomain $\Omega' \subseteq \Omega$, where $x_0$ is a
fixed point. By applying the Divergence Theorem in $\Omega'$ and
by the arbitrariness of $\Omega'$ we deduce
\begin{center}
\( {\displaystyle \left\{
\begin{array}{lr}
    {\rm div}_x T^{\Omega(s)} + qe_3=0,
    & \mathrm{in}\ \Omega,
          \vspace{0.25em}\\
    {\rm div}_x M^{\Omega} + (T^{\Omega(s)})^Te_3 \times e_3=0,
    & \mathrm{in}\ \Omega.
          \vspace{0.25em}\\
\end{array}
\right. } \) \vskip -4.5em
\begin{eqnarray}
& & \label{eq:anto-7.3}\\
& & \label{eq:anto-7.4}
\end{eqnarray}
\end{center}
Consider the case in which the boundary of the plate $\partial
\Omega$ is subjected simultaneously to a couple field
$\widehat{M}$, $\widehat{M}\cdot e_3=0$, and a transversal force
field $\widehat{Q}e_3$. Local equilibrium considerations on points
of $\partial \Omega$ yield the following boundary conditions:
\begin{center}
\( {\displaystyle \left\{
\begin{array}{lr}
    M^{\Omega}n = \widehat{M},
    & \mathrm{on}\ \partial \Omega,
          \vspace{0.25em}\\
    T^{\Omega(s)}n= \widehat{Q}e_3,
    & \mathrm{on}\ \partial \Omega.
          \vspace{0.25em}\\
\end{array}
\right. } \) \vskip -4.5em
\begin{eqnarray}
& & \label{eq:anto-7.6}\\
& & \label{eq:anto-7.7}
\end{eqnarray}
\end{center}
where $n$ is the unit outer normal to $\partial \Omega$. In
cartesian components, the equilibrium equations
\eqref{eq:anto-7.3}--\eqref{eq:anto-7.7} take the form
\begin{center}
\( {\displaystyle \left\{
\begin{array}{lr}
     M_{\alpha\beta, \beta} - Q_\alpha=0,
      & \mathrm{in}\ \Omega, \mathrm{ \alpha=1,2},
        \vspace{0.25em}\\
      Q_{\alpha,\alpha} + q=0, & \mathrm{in}\ \Omega,
          \vspace{0.25em}\\
      M_{\alpha\beta}n_\alpha n_\beta = \widehat{M}_n,
      & \mathrm{on}\ \partial \Omega,
        \vspace{0.25em}\\
      M_{\alpha\beta}\tau_\alpha n_\beta = -\widehat{M}_\tau,
      &\mathrm{on}\ \partial \Omega,
          \vspace{0.25em}\\
      Q_\alpha n_\alpha = \widehat{Q}, &\mathrm{on}\ \partial
      \Omega.
          \vspace{0.25em}\\
\end{array}
\right. } \) \vskip -8.9em
\begin{eqnarray}
& & \label{eq:anto-8.1}\\
& & \label{eq:anto-8.2}\\
& & \label{eq:anto-8.3}\\
& & \label{eq:anto-8.4}\\
& & \label{eq:anto-8.5}
\end{eqnarray}
\end{center}
Here, following a standard convention in the theory of plates, we
have decomposed the boundary couple field $\widehat{M}$ in local
coordinates as $\widehat{M}= \widehat{M}_\tau n + \widehat{M}_n
\tau$.

To complete the formulation of the equilibrium problem, we need to
introduce the constitutive equation of the material. We limit
ourselves to the Kirchhoff-Love theory and we choose to regard the
kinematical assumptions $E_{i3}[u]=0$, $i=1,2,3$ (see
\eqref{eq:anto-2.3}, with $\gamma =0$) as internal constraints,
that is we restrict the possible deformations of the points of the
plate to those whose infinitesimal strain tensor belongs to the
set
\begin{equation}
  \label{eq:anto-8.6}
  {\cal{M}}=\{E \in \M^{3\times3}|  E=E^T,  E \cdot A=0,
  \hbox{ for } A= e_i \otimes e_3 + e_3 \otimes e_i, \ i=1,2,3\}.
\end{equation}
Then, by the \textit{Generalized Principle of Determinism}
\cite{Truesdell}, the Cauchy stress tensor $T$ at any point
$(x,x_3)$ of the plate is additively decomposed in an
\textit{active} (symmetric) part $T_{A}$ and in a
\textit{reactive} (symmetric) part $T_{R}$:
\begin{equation}
  \label{eq:anto-9.1}
  T=T_{A}+T_{R},
\end{equation}
where $T_{R}$ does not work in any admissible motion, e.g., $T_{R}
\in {\cal{M}}^\perp$. Consistently with the Principle, the active
stress $T_{A}$ belongs to ${\cal{M}}$ and, in cartesian
coordinates, we have
\begin{equation}
  \label{eq:anto-9.2}
  T_{A} = T_{A \alpha\beta} e_\alpha \otimes e_\beta, \quad
  \alpha, \beta=1,2,
\end{equation}
\begin{equation}
  \label{eq:anto-9.3}
  T_{R} = T_{R \alpha 3} e_\alpha \otimes e_3 + T_{R 3\alpha} e_3 \otimes
  e_\alpha + T_{R 33} e_3 \otimes e_3.
\end{equation}
In linear theory, on assuming the reference configuration
unstressed, the active stress in a point $(x,x_3)$ of the plate,
$x \in \overline{\Omega}$ and $|x_3|\leq h/2$, is given by a
linear mapping {}from ${\cal{M}}$ into itself by means of the
fourth order \textit{elasticity tensor} $\C_{\cal{M}}$:
\begin{equation}
  \label{eq:anto-9.4}
  T_{A} = \C_{\cal{M}} E[u].
\end{equation}
Here, in view of \eqref{eq:anto-8.6} and \eqref{eq:anto-9.2},
$\C_{\cal{M}}$ can be assumed to belong to ${\mathcal L}
({\mathbb{M}}^{2}, {\mathbb{M}}^{2})$. Moreover, we assume that
$\C_{\cal{M}}$ is constant over the thickness of the plate and
satisfies the minor and major symmetry conditions expressed in
cartesian coordinates as (we drop the subscript ${\cal{M}}$)
\begin{equation}
  \label{eq:anto-9.5}
  C_{\alpha \beta \gamma \delta}=C_{\beta \alpha \gamma \delta}= C_{\alpha \beta \delta \gamma
  }=C_{\gamma \delta \alpha \beta }, \quad \alpha, \beta, \gamma,
  \delta= 1,2, \quad \hbox{in } \Omega.
\end{equation}
We refer to \cite{Podio-Guidugli} and \cite{Lembo-Podio} for a
representation formula of $\mathbb C$ based on the maximal
response symmetry of the material compatible with the internal
constraints.

Using \eqref{eq:anto-9.1} and recalling \eqref{eq:anto-5.1}, we
obtain the corresponding decomposition for $T^{\Omega}$ and
$M^{\Omega}$:
\begin{equation}
  \label{eq:anto-10.1}
    T^{\Omega}=T^{\Omega(s)}_R, \quad  M^{\Omega}=M^{\Omega}_A,
\end{equation}
that is the shear forces and the moments have reactive and active
nature, respectively. By \eqref{eq:anto-9.4}, after integration
over the thickness, the surface couple tensor is given by
\begin{equation}
  \label{eq:anto-10.2}
    M^{\Omega}(x) = - \frac{h^3}{12}  {\cal{E}}\C(x) (\nabla_x^2
    w(x)), \quad \hbox{in } \Omega,
\end{equation}
where ${\cal{E}} \in \M^2$ has cartesian components
${\cal{E}}_{11}={\cal{E}}_{11}=0$, ${\cal{E}}_{12}=-1$,
${\cal{E}}_{21}=1$. Constitutive equation \eqref{eq:anto-10.2} can
be written in more expressive way in terms of the bending and
twisting moments as follows:
\begin{equation}
  \label{eq:anto-10.3}
    M_{\alpha\beta}(w)= - P_{\alpha \beta \gamma \delta}(x)
    w,_{\gamma \delta}, \quad \alpha, \beta=1,2,
\end{equation}
where
\begin{equation}
  \label{eq:anto-10.4}
    \mathbb P(x) = \frac{h^3}{12} \C(x), \quad \hbox{in } \Omega,
\end{equation}
is the \textit{plate elasticity tensor}. Combining
\eqref{eq:anto-7.3} and \eqref{eq:anto-7.4}, and by eliminating
the reactive term $T^{\Omega(s)}_R$, we obtain the classical
partial differential equation of the Kirchhoff-Love's bending
theory of thin elastic plates, that, written in cartesian
coordinates, takes the form
\begin{equation}
  \label{eq:anto-11.1}
    (P_{\alpha \beta \gamma \delta}(x)
    w,_{\gamma \delta}(x)),_{\alpha \beta}=q, \quad \hbox{in }
    \Omega.
\end{equation}
In the remaining part of this section we complete the formulation
of the equilibrium problem for a Kirchhoff-Love plate by writing
the boundary conditions corresponding to
\eqref{eq:anto-8.3}--\eqref{eq:anto-8.5}. The determination of
these boundary conditions is not a trivial issue because, first, a
constitutive equation for shear forces is not available since
these have a reactive nature (see \eqref{eq:anto-10.1}), and,
second, because the three mechanical boundary conditions
\eqref{eq:anto-8.3}--\eqref{eq:anto-8.5} should reasonably
''collapse'' into two independent boundary conditions for the
fourth order partial differential equation \eqref{eq:anto-11.1}.
To this aim, under the additional assumption of $\C$ positive
definite, we adopt a variational approach and we impose the
stationarity condition on the \textit{total potential energy}
functional $J$ of the plate. Consider the space of regular
kinematically admissible displacements
\begin{equation}
  \label{eq:anto-11.2}
  {\cal{D}}=\{v:\Omega\times(-h/2,h/2) \rightarrow \R^3 | \ v(x,x_3)=\eta(x)e_3 -x_3 \nabla_x\eta(x), \hbox{ with } \eta:\Omega \rightarrow \R \}.
\end{equation}
The energy functional $J: {\cal{D}} \rightarrow \R$ is defined as
\begin{equation}
  \label{eq:anto-11.3}
  J(v)=a(v,v)-l(v),
\end{equation}
where $a(v,v)$ is interpreted as the \textit{elastic energy}
stored in the plate for the displacement field $v$ and $l(v)$ is
the \textit{load potential} that accounts for the energy of the
system of applied loads $q$, $\widehat{M}$, $\widehat{Q}$. The
three-dimensional expression of the elastic energy in the Linear
Theory of Elasticity is given by
\begin{equation}
  \label{eq:anto-12.1}
  a(v,v)=\frac{1}{2} \int_{-h/2}^{h/2} \int_\Omega \C E[v] \cdot
  E[v] dx dx_3,
\end{equation}
where, in view of \eqref{eq:anto-11.2}, $E[v]=-x_3
\nabla_x^2\eta(x)$. After integration over the thickness, we
obtain
\begin{equation}
  \label{eq:anto-12.2}
  a(v,v)=- \frac{1}{2} \int_\Omega M_{\alpha\beta}(\eta) \eta,_{\alpha\beta}
  dx,
\end{equation}
where $M_{\alpha\beta}(\eta)$ are as in \eqref{eq:anto-10.3} with
$w$ replaced by $\eta$. The load functional has the expression
\begin{equation}
  \label{eq:anto-12.3}
  l(v)=\int_\Omega q\eta dx + \int_{\partial\Omega}
  (\widehat{Q}\eta + \widehat{M}_2 \eta,_2 - \widehat{M}_1
  \eta,_1) ds .
\end{equation}
Then, the stationarity condition (in fact, minimum condition) on
$J$ at $w$ yields
\begin{equation}
  \label{eq:anto-13.1}
  \int_\Omega M_{\alpha\beta}(w)\eta,_{\alpha\beta} dx +
  \int_\Omega q\eta + \int_{\partial\Omega}
  (\widehat{Q}\eta + \widehat{M}_2 \eta,_2 - \widehat{M}_1
  \eta,_1) ds=0,
\end{equation}
for every regular function $\eta$. Integrating by parts twice on
the first integral we obtain
\begin{multline}
  \label{eq:anto-13.2}
  \int_\Omega (M_{\alpha\beta,\alpha\beta }(w)+q)\eta dx +
  \int_{\partial\Omega} ( - M_{\alpha\beta,\beta }(w)n_\alpha +
  \widehat{Q})\eta  ds + \\
  + \int_{\partial\Omega} (M_{\alpha\beta}(w)n_\beta \eta_\alpha + \widehat{M}_2 \eta,_2 - \widehat{M}_1
  \eta,_1) ds=0.
\end{multline}
We elaborate the last integral $I_1$ of \eqref{eq:anto-13.2} by
rewriting the first order derivatives of $\eta$ on $\partial
\Omega$ in terms of the normal and arc-length derivative of
$\eta$. We have
\begin{multline}
  \label{eq:anto-13.3}
  I_1= \int_{\partial\Omega} (M_{\alpha\beta}(w)n_\beta n_\alpha - \widehat{M}_1 n_1 + \widehat{M}_2 n_2)\eta,_n ds +
  \\
  + \int_{\partial\Omega} (M_{\alpha\beta}(w)n_\beta \tau_\alpha + \widehat{M}_1 n_2 + \widehat{M}_2 n_1)\eta,_s
  ds = I_1' + I_1''
\end{multline}
and, integrating by parts on $\partial \Omega$, we get
\begin{multline}
  \label{eq:anto-13.4}
  I_1'= (M_{\alpha\beta}(w)n_\beta \tau_\alpha + \widehat{M}_1 n_2 + \widehat{M}_2
  n_1)\eta|_{s=0}^{s={{ {l}}(\partial \Omega)}}-
  \\
  - \int_{\partial\Omega} (M_{\alpha\beta}(w)n_\beta \tau_\alpha + \widehat{M}_1 n_2 + \widehat{M}_2
  n_1),_s \eta ds,
\end{multline}
where $ { {l}}(\partial \Omega)$ is the length of $\partial
\Omega$. Since $\partial \Omega$ is of class $C^{1,1}$, the
boundary term on the right end side of \eqref{eq:anto-13.4}
identically vanishes. Therefore, the stationarity condition of $J$
at $w$ takes the final form
\begin{multline}
  \label{eq:anto-14.1}
    \int_\Omega (M_{\alpha\beta,\alpha\beta }(w)+q)\eta dx + \\
    \int_{\partial\Omega}
    \left (
    -(M_{\alpha\beta}(w)n_\beta \tau_\alpha),_s - M_{\alpha\beta,\beta
    }(w)n_\alpha + \widehat{Q} - (\widehat{M}_1 n_2 + \widehat{M}_2
  n_1),_s
  \right )\eta ds + \\
  +
    \int_{\partial\Omega}
    \left (
    M_{\alpha\beta}(w)n_\beta n_\alpha - \widehat{M}_1 n_1 + \widehat{M}_2 n_2
    \right )
    \eta,_n ds =0
\end{multline}
for every $\eta \in C^\infty( \overline{\Omega},\R) $. By the
arbitrariness of the function $\eta$, and of the traces of $\eta$
and $\eta,_n$ on $\partial \Omega$, we determine the equilibrium
equation \eqref{eq:anto-11.1} and the desired Neumann boundary
conditions on $\partial \Omega$:
\begin{equation}
  \label{eq:anto-14.2}
  M_{\alpha\beta}(w)n_\alpha n_\beta = \widehat{M}_n,
\end{equation}
\begin{equation}
  \label{eq:anto-14.3}
  M_{\alpha\beta,\beta}(w)n_\alpha + (M_{\alpha\beta}(w)n_\beta \tau_\alpha),_s = \widehat{Q} -(\widehat{M}_\tau),_s.
\end{equation}

\section{Uniqueness and stability of extreme inclusions and free boundaries} \label{sec:uniq_stab}
\subsection{Rigid inclusions: uniqueness} \label{subsec:rigid_uniq}

In the sequel we shall assume that the plate is made of nonhomogeneous linear
elastic material with plate tensor
\begin{equation}
\label{eq:P_def}
   \mathbb{P}=\frac{h^3}{12}\mathbb{C},
\end{equation}
where the elasticity tensor $\C(x) \in{\cal L}
({\M}^{2}, {\M}^{2})$ has cartesian components
$C_{ijkl}$ which satisfy the following symmetry conditions
\begin{equation}
  \label{eq:sym-conditions-C-components}
    C_{ijkl} = C_{ klij} =
    C_{ klji} \quad i,j,k,l
    =1,2, \hbox{ a.e. in } \Omega,
\end{equation}
and, for simplicity, is defined in all of $\R^2$.

We make the following assumptions:

{I)} \textit{Regularity}
\begin{equation}
  \label{eq:3.bound}
  \mathbb{C} \in C^{1,1}(\R^2,   {\mathcal L} ({\mathbb{M}}^{2},
  {\mathbb{M}}^{2})),
\end{equation}

{II)} \textit{Ellipticity (strong convexity)} There exists $\gamma>0$
such that

\begin{equation}
  \label{eq:3.convex}
    {\mathbb{C}}A \cdot A \geq  \gamma |A|^2, \qquad \hbox{in } \R^2,
\end{equation}
for every $2\times 2$ symmetric matrix $A$.

{III)} \textit{Dichotomy condition}
\begin{subequations}
\begin{eqnarray}
\label{3.D(x)bound} either &&  {\mathcal{D}}(x)>0,
\quad\hbox{for every } x\in \R^2, \\[2mm]
\label{3.D(x)bound 2} or && {\mathcal{D}}(x)=0, \quad\hbox{for every }
x\in \R^2,
\end{eqnarray}
\end{subequations}
where
\begin{equation}
    \label{3.D(x)}
    {\mathcal{D}}(x)= \frac{1}{a_0} |\det S(x)|,
\end{equation}
\begin{equation}
    \label{3. S(x)}
    S(x) = {\left(
\begin{array}{ccccccc}
  a_0   & a_1   & a_2   & a_3   & a_4   & 0    &    0    \\
  0     & a_0   & a_1   & a_2   & a_3   & a_4  &    0    \\
  0     & 0     & a_0   & a_1   & a_2   & a_3  &    a_4  \\
  4a_0  & 3a_1  & 2a_2  & a_3   & 0     & 0    &    0    \\
  0     & 4a_0  & 3a_1  & 2a_2  & a_3   & 0    &    0    \\
  0     & 0     & 4a_0  & 3a_1  & 2a_2  & a_3  &    0    \\
  0     & 0     & 0     & 4a_0  & 3a_1  & 2a_2 &    a_3  \\
  \end{array}
\right)},
\end{equation}
\begin{equation}
    \label{3.coeffsmall}
    a_0=A_0, \ a_1=4C_0, \ a_2=2B_0+4E_0, \ a_3=4D_0, \ a_4=F_0.
\end{equation}
and
\begin{center}
\( {\displaystyle \left\{
\begin{array}{lr}
    C_{1111}=A_0, \ \ C_{1122}=C_{2211}=B_0,
         \vspace{0.12em}\\
    C_{1112}=C_{1121}=C_{1211}=C_{2111}=C_0,
        \vspace{0.12em}\\
    C_{2212}=C_{2221}=C_{1222}=C_{2122}=D_0,
        \vspace{0.12em}\\
    C_{1212}=C_{1221}=C_{2112}=C_{2121}=E_0,
        \vspace{0.12em}\\
    C_{2222}=F_0,
        \vspace{0.25em}\\
\end{array}
\right. } \) \vskip -3.0em
\begin{eqnarray}
\ & & \label{3.coeff6}
\end{eqnarray}
\end{center}

\begin{rem}
  \label{rem:dichotomy} Whenever \eqref{3.D(x)bound} holds we denote
\begin{equation}
    \label{delta-1}
    \delta_1=\min_{\R^2}{\mathcal{D}}.
\end{equation}
We emphasize that, in all the following statements, whenever a
constant is said to depend on $\delta_1$ (among other quantities) it
is understood that such dependence occurs \textit{only} when
\eqref{3.D(x)bound} holds.
\end{rem}

On the assigned couple field $\widehat M$ let us require the
following assumptions:
\begin{equation}
  \label{eq:3.emme_reg}
  \widehat M\in H^{-\frac{1}{2}}(\partial\Omega,\mathbb{R}^2), \quad (\widehat M_n,\widehat M_{\tau,s}) \not\equiv 0,
\end{equation}
\begin{equation}
  \label{eq:3.emme_comp}
  \int_{\partial\Omega}\widehat M_i=0,\qquad i=1,2.
\end{equation}

\begin{theo} [Unique determination of a rigid inclusion with one measurement]
\label{theo:main-unknown-rig-incl}
Let $\Omega$ be a simply connected domain in $\mathbb{R}^2$ such
that $\partial\Omega$ is of class $C^{1,1}$ and let $D_i$,
$i=1,2$, be two simply connected domains compactly contained in
$\Omega$, such that $\partial D_i$ is of class $C^{3,1}$, $i=1,2$.
Moreover, let $\Gamma$ be a nonempty open portion of
$\partial\Omega$, of class $C^{3,1}$. Let the plate tensor
$\mathbb{P}$ be given by \eqref{eq:P_def}, and satisfying
\eqref{eq:sym-conditions-C-components}--\eqref{eq:3.convex} and
the dichotomy condition \eqref{3.D(x)bound} or \eqref{3.D(x)bound
2}. Let $\widehat M$ be a boundary couple field satisfying
\eqref{eq:3.emme_reg}--\eqref{eq:3.emme_comp}. Let $w_i$, $i=1,2$,
be the solutions to the mixed problem
\eqref{eq:1.dir-pbm-incl-rig-1bis}--\eqref{eq:1-dir-pbm-incl-rig-5bis},
coupled with \eqref{eq:1.equil-rigid-incl}, with $D=D_i$.

If there exists $g \in \mathcal{A} $ such that
\begin{equation}
    \label{eq:3.theorem2-hp}
w_1 - w_2 = g,\quad (w_1 - w_2)_{,n} =  g_{,n}, \quad \hbox{on
}\Gamma,
\end{equation}
then
\begin{equation}
    \label{eq:3.theorem2-tesi}
    D_1=D_2.
\end{equation}
\end{theo}

\begin{proof}[Proof of Theorem \ref{theo:main-unknown-rig-incl}]

Let $G$ be the connected component of
$\Omega\setminus(\overline{D_1\cup D_2})$ such that
$\Gamma\subset\partial G$. Let us notice that, since $w_i$ satisfies homogeneous Dirichlet conditions on
the $C^{3,1}$ boundary $\partial D_i$, by regularity results we
have that $w_i \in H^4(\widetilde{\Omega}  \setminus D_i)$, for
every $\widetilde{\Omega}$, $D_i \subset\subset
{\widetilde{\Omega} } \subset\subset \Omega$, $i=1,2$ (see, for
example, \cite{l:agmon}). By Sobolev embedding theorems (see, for
instance, \cite{l:ad}), we have that $w_i$ and $\nabla w_i$ are
continuous up to $\partial D_i$, $i=1,2$. Therefore
\begin{equation}
  \label{eq:4.null_vi}
  w_i\equiv 0, \quad \nabla w_i^e\equiv 0, \quad \hbox{on }\partial D_i.
\end{equation}
Let $w=w_1-w_2-g$, with $g(x_1,x_2)=ax_1+bx_2+c$. By our
assumptions, $w$ takes homogeneous Cauchy data on $\Gamma$. {}From
the uniqueness of the solution to the Cauchy problem (see, for
instance, Theorem $3.8$ in \cite{M-R-V4}) and also Remark 4 in
\cite{M-R-V2}) and {}from the weak unique continuation property,
we have that
\begin{equation*}
\qquad\qquad w\equiv 0, \quad\hbox{in }G.
\end{equation*}
Let us prove, for instance, that $D_2\subset D_1$.
We have
\begin{equation*}
  D_2\setminus \overline{D_1}\subset\Omega\setminus(\overline{D_1\cup
  G}),
\end{equation*}

\begin{equation*}
  \partial(\Omega\setminus(\overline{D_1\cup G}))=\Sigma_1\cup\Sigma_2,
\end{equation*}
where $\Sigma_2=\partial D_2\cap\partial G$,
$\Sigma_1=\partial(\Omega\setminus(\overline{D_1\cup G}))\setminus \Sigma_2\subset\partial D_1$.

Let us distinguish two cases

\medskip

i) $\partial D_1\cap\Sigma_2\neq\emptyset$,

ii) $\partial D_1\cap\Sigma_2=\emptyset$.

In case i), there exists $ P_0 \in \partial D_1\cap\Sigma_2$. Since $w_i(P_0)=0$, $w(P_0)=0$, we have $g(P_0)=0$.

Let $P_n\in G$, $ P_n\rightarrow P_0$. We have

\begin{equation*}
  \nabla w(P_n)=0,
\end{equation*}
\begin{equation*}
  0=\lim_{n\rightarrow \infty}\nabla w(P_n)=\nabla w_1^e(P_0)-\nabla w_2^e(P_0)-(a,b)=-(a,b),
\end{equation*}
so that $g\equiv c$,
and, since $g(P_0)=0$,
\begin{equation}
   \label{eq:3.tutti_nulli}
  g\equiv 0 \quad\Rightarrow \quad w_1\equiv w_2, \ \hbox{in } G, \quad\Rightarrow \quad
  \left\{
\begin{array}{lr}
 w_1\equiv w_2=0, \ \hbox{on } \Sigma_2,
    \vspace{0.25em}\\
  \nabla w_1\equiv  \nabla w_2=0, \ \hbox{on } \Sigma_2 .
\end{array}
\right.
\end{equation}

Integrating by parts equation $({\rm div}({\rm div} ({\mathbb
P}\nabla^2 w_1)))w_1=0$, we obtain
\begin{multline}
   \label{eq:3.by_parts_Diri}
  \int_{\Omega\setminus(\overline{D_1\cup G})}\mathbb{P}\nabla^2w_1 \cdot \nabla^2w_1=
  \int_{\Sigma_1\cup\Sigma_2} (\mathbb{P}\nabla^2w_1 \nu\cdot \nu) w_{1,n} +\\
  + \int_{\Sigma_1\cup\Sigma_2} (\mathbb{P}\nabla^2w_1 \nu\cdot \tau) w_{1,s}
  - \int_{\Sigma_1\cup\Sigma_2} (\mathrm{div}(\mathbb{P}\nabla^2w_1)\cdot\nu )w_1,
\end{multline}
where $\nu$ is the outer unit normal to $\Omega\setminus(\overline{D_1\cup G})$.
By
\begin{equation*}
  w_1=0,\ \nabla w_1=0, \hbox{ on }\Sigma_1,
\end{equation*}
and by \eqref{eq:3.tutti_nulli}, we have
\begin{equation*}
  w_1=w_2=0,\ \nabla w_1= \nabla w_2=0, \hbox{ on }\Sigma_2,
\end{equation*}
so that
\begin{equation}
   \label{3.null_int}
0=\int_{\Omega\setminus(\overline{D_1\cup
G})}\mathbb{P}\nabla^2w_1 \cdot \nabla^2w_1\geq \gamma
\int_{D_2\setminus \overline{D_1}} |\nabla^2w_1|^2.
\end{equation}
If $D_2\setminus \overline{D_1}\neq \emptyset$, then, by the weak unique continuation principle, $w_1$ coincides with an affine function in
$\Omega\setminus\overline{D_1}$, contradicting the choice of the nontrivial Neumann data $\widehat{M}$ on $\partial\Omega$.
Therefore $D_2\subset \overline{D_1}$ and, by the
regularity of $D_i$, $i=1,2$, $D_2\subset D_1$.

In case ii),
either $\overline{D_1}\cap \overline{D_2} =\emptyset$ or $\overline{D_1}\subset D_2$.

Let us consider for instance the first case, the proof of the
second case being similar. Integrating by parts, we have
\begin{multline*}
  \int_{D_2}\mathbb{P}\nabla^2w_1 \cdot  \nabla^2w_1=\int_{D_2}\mathbb{P}\nabla^2w_1 \cdot  \nabla^2(w_1-g)=\\
  \int_{\partial D_2} (\mathbb{P}\nabla^2w_1 \nu\cdot \nu) (w_1-g),_n +\\
  + \int_{\partial D_2} (\mathbb{P}\nabla^2w_1 \nu\cdot \tau) (w_1-g),_s
  - \int_{\partial D_2} (\mathrm{div}(\mathbb{P}\nabla^2w_1)\cdot\nu )(w_1-g).
\end{multline*}
By the regularity of $\partial D_2$, we may rewrite it as
\begin{multline*}
  \int_{D_2}\mathbb{P}\nabla^2w_1 \cdot \nabla^2w_1=
  \int_{\partial D_2} ( \mathbb{P}\nabla^2w_1 \nu\cdot \nu )(w_1-g),_n + \\
  -\int_{\partial D_2} \left((\mathbb{P}\nabla^2w_1 \nu\cdot \tau),_s
  +  \mathrm{div}(\mathbb{P}\nabla^2w_1)\cdot\nu\right) (w_1-g),
\end{multline*}

Recalling that $w_1-g=w_2=0$, $(w_1-g),_n=w_{2,n}=0$ on $\partial D_2$, we have that
\begin{equation*}
  \int_{D_2}\mathbb{P}\nabla^2w_1 \cdot \nabla^2w_1=0.
\end{equation*}

If $D_2\neq\emptyset$, then  $w_1$ coincides with an affine function in $D_2$, and, by the weak unique continuation principle, also in
$\Omega\setminus \overline{D_1}$, contradicting the choice of a nontrivial $\widehat{M}$.
Therefore $D_2=\emptyset$.
Symmetrically, we obtain that $D_1=\emptyset$, that is $D_1=D_2$.

\end{proof}

\subsection{Rigid inclusions: stability} \label{subsec:stab}

In order to prove the stability estimates, we need the following further quantitative assumptions.

Given $\rho_0$, $M_0$, $M_1>0$, we assume that
\begin{equation}
   \label{eq:bound_area}
|\Omega|\leq M_1\rho_0^2,
\end{equation}
\begin{equation}
   \label{eq:compactness}
    \hbox{dist}(D, \partial \Omega) \geq \rho_0,
\end{equation}
\begin{equation}
   \label{eq:reg_Omega}
\partial\Omega \hbox{ is of } class\  C^{2,1}
\ with\ constants\ \rho_0, M_0,
\end{equation}
\begin{equation}
   \label{eq:reg_Sigma}
\Gamma \hbox{ is of } class\  C^{3,1} \ with\ constants\
\rho_0, M_0,
\end{equation}
\begin{equation}
   \label{eq:reg_D}
\partial D \hbox{ is of } class\  C^{3,1}
\ with\ constants\ \rho_0, M_0,
\end{equation}
where $|\Omega|$ denotes the Lebesgue measure of $\Omega$.
Moreover, we assume that for some $P_0\in\Sigma$ and some $\delta_0$, $0<\delta_0<1$,
\begin{equation}
   \label{eq:large_enough}
   \partial\Omega\cap
R_{\frac{\rho_0}{M_0},\rho_0}(P_0)\subset\Gamma,
\end{equation}
and that
\begin{equation}
   \label{eq:small_enough}
   |\Gamma|\leq(1-\delta_0)|\partial\Omega|.
\end{equation}

On the Neumann data $\widehat{M}$ we assume that
\begin{equation}
   \label{eq:reg_M}
\widehat{M}\in L^2(\partial \Omega,\R^2),\quad
(\widehat{M}_n,(\widehat{M}_\tau),_s)\not\equiv 0,
\end{equation}
\begin{equation}
   \label{eq:M_comp}
    \int_{\partial \Omega} \widehat{M}_i = 0, \quad i=1,2,
\end{equation}
\begin{equation}
   \label{eq:supp_M}
\hbox{supp}(\widehat{M})\subset\subset\Gamma,
\end{equation}
and that, for a given constant $F>0$,
\begin{equation}
\label{eq:M_frequency}
   \frac{\|\widehat{M}\|_{L^2(\partial
   \Omega ,\R^2)}}{ \|\widehat{M}\|_{H^{-\frac{1}{2}}(\partial \Omega,\R^2)}}\leq
   F.
\end{equation}

On the elasticity tensor $\mathbb C$, we assume the same a priori
information made in Subsection \ref{subsec:rigid_uniq}, and we
introduce a parameter $M>0$ such that
\begin{equation}
  \label{eq:3.bound_quantit}
  \sum_{i,j,k,l=1}^2 \sum_{m=0}^2 \rho_0^m \|\nabla^m C_{ijkl}\|_{L^\infty(\R^2)} \leq
    M.
\end{equation}

We shall refer to the set of constants $M_0$, $M_1$, $\delta_0$,
$F$, $\gamma$, $M$, $\delta_1$ as the \emph{a priori data}. The
scale parameter $\rho_0$ will appear explicitly in all formulas,
whereas the dependence on the thickness parameter $h$ will be
omitted.

\begin{theo}[Stability result]
  \label{theo:Main}
Let $\Omega$ be a bounded domain in $\R^2$ satisfying
\eqref{eq:bound_area} and \eqref{eq:reg_Omega}. Let $D_i$,
$i=1,2$, be two simply connected open subsets of $\Omega$
satisfying \eqref{eq:compactness} and \eqref{eq:reg_D}. Moreover,
let $\Gamma$ be an open portion of $\partial\Omega$ satisfying
\eqref{eq:reg_Sigma}, \eqref{eq:large_enough} and
\eqref{eq:small_enough}. Let $\widehat{M}\in
L^2(\partial\Omega,\R^2)$ satisfy
\eqref{eq:reg_M}--\eqref{eq:M_frequency} and let the plate tensor
$\mathbb{P}$ given by \eqref{eq:P_def} satisfy
\eqref{eq:sym-conditions-C-components},
\eqref{eq:3.bound_quantit}, \eqref{eq:3.convex} and the dichotomy
condition. Let $w_i\in H^2(\Omega \setminus \overline{D_i})$ be
the solution to
\eqref{eq:1.dir-pbm-incl-rig-1bis}--\eqref{eq:1-dir-pbm-incl-rig-5bis},
coupled with \eqref{eq:1.equil-rigid-incl}, when $D=D_i$, $i=1,2$.
If, given $\epsilon>0$, we have
\begin{equation}
    \label{eq:small_L2}
\min_{g \in \cal{A}} \left\{\|w_1 - w_2 -g \|_{L^2(\Sigma)}+\rho_0
\left\| (w_1 - w_2 -g)_{,n} \right\|_{L^2(\Sigma)}\right\}\leq
\epsilon,
\end{equation}
then we have
\begin{equation}
    \label{eq:small_Haus_bound}
d_{\cal H}(\partial D_1,\partial D_2) \leq C\rho_0 (\log|\log
\widetilde{\epsilon}|)^{-\eta}, \quad  0< \widetilde{\epsilon}
<e^{-1},
\end{equation}
and
\begin{equation}
    \label{eq:small_Haus_inclusion}
d_{\cal H}( \overline{D_1},\overline{D_2} ) \leq C\rho_0
(\log|\log \widetilde{\epsilon}|)^{-\eta}, \quad \
0<\widetilde{\epsilon}<e^{-1},
\end{equation}
where $\widetilde{\epsilon}=\frac{\epsilon}
{\rho_0^2\|\widehat{M}\|_{H^{-\frac{1}{2}}}}$ and  $C$, $\eta$, $C>0$, $0<\eta\leq 1$, are constants only
depending on the a priori data.
\end{theo}

\begin{proof} [Proof of Theorem ~\ref{theo:Main}]
Let us rough-out a sketch of the proof, referring the interested
reader to Section 3 of \cite{M-R-V5}. Retracing the proof of the
uniqueness theorem, the basic idea is that of deriving the
quantitative version in the stability context of the vanishing of
the integral $\int_{\Omega\setminus(\overline{D_i\cup
G})}\mathbb{P}\nabla^2w_i \cdot \nabla^2w_i$, for $i=1,2$, that is
a control with some small term emerging {}from the bound
\eqref{eq:small_L2} on the Cauchy data. To this aim, we need
stability estimates of continuation {}from Cauchy data and
propagation of smallness estimates. In particular, the propagation
of smallness of $|\nabla^2 w_i|$, for $i=1,2$, {}from a
neighboorhood of $\Gamma$ towards
$\partial(\Omega\setminus(\overline{D_i\cup G}))$ is performed
through iterated application of the three sphere inequality
\eqref{eq:3sph} over suitable chains of disks. A strong hindrance
which occurs in this step is related to the difficulty of getting
arbitrarily closer to the boundary of the set
$\Omega\setminus(\overline{D_i\cup G})$, due to the absence, in
our general setting, of any a priori information on the reciprocal
position of $D_1$ and $D_2$. For this reason, as a preparatory
step, we derive the following rough estimate

\begin{equation}
   \label{eq:Cauchy1}
\max\left\{
\int_{D_2\setminus \overline{D_1}}|\nabla^2 w_1|^2,
\int_{D_1\setminus \overline{D_2}}|\nabla^2 w_2|^2\right\}
\leq
\rho_0^2\|\widehat{M}\|_{H^{-\frac{1}{2}}}^2
(\log|\log \widetilde{\epsilon}|)^{-\frac{1}{2}},
\end{equation}
which holds for every
$\widetilde{\epsilon}<e^{-1}$, with $\widetilde{\epsilon}=\frac{\epsilon}
{\rho_0^2\|\widehat{M}\|_{H^{-\frac{1}{2}}}}$,
and where $C>0$ depends only on $\gamma$, $M$, $\delta_1$,
$M_0$, $M_1$ and $\delta_0$.

Since a pointwise lower bound for $|\nabla^2 w_i|^2$ cannot hold in general,
the next crucial step consists in the following Claim.

\noindent
\textbf{Claim.} If
\begin{equation}
   \label{eq:eta}
\max\left\{
\int_{D_2 \setminus \overline{D_1}}|\nabla^2 w_1|^2, \int_{D_1 \setminus \overline{D_2}}|\nabla^2 w_2|^2\right\} \leq\ \frac{\eta}{\rho_0^2},
\end{equation}
then
\begin{equation}
   \label{eq:d<eta}
d_{\cal H}(\partial D_1,\partial
D_2)\leq
C\rho_0\left[\log\left(\frac{C\rho_0^4\|\widehat{M}\|_{H^{-\frac{1}{2}}(\partial
\Omega ,\R^n)}^2} {\eta}\right)\right]^{-\frac{1}{B}},
\end{equation}
where $B>0$ and $C>0$ only depend on $\gamma$, $M$, $\delta_1$,
$M_0$, $M_1$, $\delta_0$ and $F$.
\begin{proof} [Proof of the Claim]
Denoting for simplicity $d=d_{\cal H}(\partial D_1,\partial
D_2)$,
we may assume, with no loss of generality, that there exists
$x_0\in \partial D_1$ such that $\hbox{dist}(x_0,\partial D_2)=d$.
Let us distinguish two cases:

\item{i)} $B_d(x_0) \subset D_2$;
\item{ii)} $B_d(x_0) \cap D_2 = \emptyset$.

In case i), by the regularity assumptions made on $\partial D_1$,
there exists $x_1\in D_2 \setminus D_1$ such that
$B_{td}(x_1)\subset D_2 \setminus D_1$, with
$t=\frac{1}{1+\sqrt{1+M_0^2}}$.

In \cite{M-R-V6}, by iterated application of the three sphere inequality \eqref{eq:3sph},
we have obtained the following \emph{Lipschitz propagation of smallness} estimate:
there exists $s>1$, only depending on $\gamma$, $M$, $\delta_1$, $M_0$
and $\delta_0$, such that for every $\rho>0$ and every $\bar x\in
(\Omega\setminus \overline{D})_{s\rho}$, we have
\begin{equation}
   \label{eq:LPS}
\int_{B_\rho(\bar x)}|\nabla^2 w|^2\geq
\frac{C\rho_0^2}{\exp\left[A\left(\frac{\rho_0}{\rho}\right)^B\right]}
\| \widehat{M} \|_{H^{-\frac{1}{2}}}^2,
\end{equation}
where $A>0$, $B>0$ and $C>0$ only depend on $\gamma$, $M$,
$\delta_1$, $M_0$, $M_1$, $\delta_0$ and $F$. By \eqref{eq:eta}
and by applying \eqref{eq:LPS} with $\rho=\frac{td}{s}$, we have
\begin{equation}
   \label{eq:eta>}
\eta\geq\frac{C\rho_0^4}{\exp{\left[A\left(\frac{s\rho_0}{td}\right)^B\right]}}
\|\widehat{M}\|_{H^{-\frac{1}{2}}(\partial \Omega ,\R^2)}^2,
   \end{equation}
where $A>0$, $B>0$ and $C>0$ only depend on $\gamma$, $M$, $\delta_1$,
$M_0$, $M_1$, $\delta_0$ and $F$.

By \eqref{eq:eta>} we easily find \eqref{eq:d<eta}.

Case ii) can be treated similarly by substituting $w_1$ with
$w_2$.
\end{proof}
By applying the Claim to \eqref{eq:Cauchy1}, we obtain a first stability estimates
of log-log-log type.
At this stage, a tool which turns out to be very useful is a geometrical result,
firstly stated in \cite{A-B-R-V}, which ensures that there exists $\epsilon_0>0$, only
depending on $\gamma$, $M$, $\delta_1$, $M_0$, $M_1$, $\delta_0$
and $F$, such that if $\epsilon\leq \epsilon_0$ then $\partial G$
is of Lipschitz class with constants $\widetilde{\rho_0}$, $L$,
with $L$ and $\frac{\widetilde{\rho_0}}{\rho_0}$ only depending on
$M_0$. Lipschitz regularity prevents the occurrence of uncontrollable narrowings or cuspidal points in $G$
and allows to refine the geometrical constructions of the chains of disks to which we apply the three sphere inequality, obtaining the better estimate
\begin{equation}
   \label{eq:Cauchy2}
\max\left\{
\int_{D_2\setminus \overline{D_1}}|\nabla^2 w_1|^2,
\int_{D_1\setminus \overline{D_2}}|\nabla^2 w_2|^2\right\}
\leq
\rho_0^2\|\widehat{M}\|_{H^{-\frac{1}{2}}}^2
|\log \widetilde{\epsilon}|^{-\sigma},
\end{equation}
which holds for every $\widetilde{\epsilon}<1$, where $C>0$ and
$\sigma >0$  depend only on $\gamma$, $M$, $\delta_1$, $M_0$,
$M_1$, $\delta_0$, $L$ and $\frac{\widetilde{\rho_0}}{\rho_0}$.
Again, by applying the Claim, the desired estimates follow.
\end{proof}

\subsection{Cavities and unknown boundary portions: uniqueness} \label{subsec:cavities_uniq}
In this subsection we consider the inverse problems of determining unknown boundaries of the following
two kinds: i) the boundary of a cavity, ii) an unknown boundary portion of $\partial\Omega$.
In both cases, we have homogeneous Neumann conditions on the unknown boundary. Neumann boundary conditions
lead to further complications in the arguments involving integration by parts.
To give an idea of the differences, given any connected component $F$ of $D_2\setminus \overline{D_1}$,
whose boundary is made of two arcs $\tau \subset \partial D_1$ and $\gamma\subset \partial D_2$, having common endpoints $P_1$ and $P_2$,
the analogue of \eqref{eq:3.by_parts_Diri} becomes
\begin{multline}
   \label{eq:3.by_parts_Neu}
  \int_{F}\mathbb{P}\nabla^2w_1 \cdot  \nabla^2w_1=
  \int_{\tau} (\mathbb{P}\nabla^2w_1 n^1\cdot \tau^1 w_{1})_{,s} -
  \int_{\gamma} (\mathbb{P}\nabla^2w_2 n^2\cdot \tau^2 w_{2})_{,s}=\\
  =\left[(\mathbb{P}\nabla^2w_1 n^1\cdot \tau^1)(P_1)-
  (\mathbb{P}\nabla^2w_1 n^2\cdot \tau^2)(P_1)\right]
  (w_1(P_1)-w_1(P_2)),
\end{multline}
\noindent where $n^i$, $\tau^i$ denotes the unit normal and
tangent vector to $ D_i$, $i=1,2$.

Since, in general, the boundaries of $D_1$ and $D_2$ intersect
nontangentially, the above expression does not vanish and the
contradiction arguments fails. For this reason, we need two
boundary measurements to prove uniqueness, as stated in Theorem
\ref{theo:uniq_cavities}. Instead, uniqueness with one measurement
can be restored in the problem of the determination of an unknown
boundary portion, by taking advantage of the fact that the two
plates have a common regular boundary portion, say $\Gamma$, see
Theorem \ref{theo:main-unknown-bdry}.

\medskip
\noindent
Let $D\subset\subset\Omega$ be a domain of class $C^{1,1}$ representing an unknown cavity inside the plate $\Omega$.
Under the same assumptions made in Subsection \ref{subsec:rigid_uniq}, the transversal displacement $w$ satisfies the
following Neumann problem

\begin{center}
\( {\displaystyle \left\{
\begin{array}{lr}
     {\rm div}({\rm div} (
      {\mathbb P}\nabla^2 w))=0,
      & \mathrm{in}\ \Omega \setminus \overline {D},
        \vspace{0.25em}\\
      ({\mathbb P} \nabla^2 w)n\cdot n=-\widehat M_n, & \mathrm{on}\ \partial \Omega,
          \vspace{0.25em}\\
      {\rm div}({\mathbb P} \nabla^2 w)\cdot n+(({\mathbb P} \nabla^2
      w)n\cdot \tau),_s
      =(\widehat M_\tau),_s, & \mathrm{on}\ \partial \Omega,
        \vspace{0.25em}\\
      ({\mathbb P} \nabla^2 w)n\cdot n=0, & \mathrm{on}\ \partial D,
        \vspace{0.25em}\\
    {\rm div}({\mathbb P} \nabla^2 w)\cdot n+(({\mathbb P} \nabla^2
        w)n\cdot \tau),_s=0, & \mathrm{on}\ \partial D,
          \vspace{0.25em}\\
\end{array}
\right. } \) \vskip -8.9em
\begin{eqnarray}
& & \label{eq:dir-pbm-unkn-bdry-1}\\
& & \label{eq:dir-pbm-unkn-bdry-2}\\
& & \label{eq:dir-pbm-unkn-bdry-3}\\
& & \label{eq:dir-pbm-unkn-bdry-4}\\
& & \label{eq:dir-pbm-unkn-bdry-5}
\end{eqnarray}

\end{center}
\noindent
which admits a solution $w \in H^{2}(\Omega\setminus \overline{\Omega})$, which is uniquely
determined up to addition of an affine function.

Concerning the inverse problem of the determination of the cavity
$D$ inside the plate, let us recall the following result.
\begin{theo} [Uniqueness with two boundary measurements]
\label{theo:uniq_cavities} Let $\Omega$ be a simply connected
domain in $\mathbb{R}^2$ such that $\partial\Omega$ is of class
$C^{1,1}$ and let $D_i$, $i=1,2$, be two simply connected domains
compactly contained in $\Omega$, such that $\partial D_i$ is of
class $C^{4,1}$, $i=1,2$. Moreover, let $\Gamma$ be a nonempty
open portion of $\partial\Omega$, of class $C^{3,1}$. Let the
plate tensor $\mathbb{P}$ be given by \eqref{eq:P_def}, and
satisfying \eqref{eq:3.bound} \eqref{eq:3.convex} and the
dichotomy condition \eqref{3.D(x)bound} or \eqref{3.D(x)bound 2}.
Let $\widehat M$, $\widehat M^*$ be two boundary couple fields
both satisfying \eqref{eq:3.emme_reg}--\eqref{eq:3.emme_comp} and
such that $(\widehat M_n, \widehat M_{\tau,s})$ and $(\widehat
M_n^*, \widehat M_{\tau,s}^*)$ are linearly independent in
$H^{-\frac{1}{2}}(\partial\Omega, \mathbb{R}^2)\times
H^{-\frac{3}{2}}(\partial\Omega, \mathbb{R}^2)$. Let $w_i$,
$w_i^*$, $i=1,2$, be solutions to the Neumann problem
\eqref{eq:dir-pbm-unkn-bdry-1}--\eqref{eq:dir-pbm-unkn-bdry-5},
with $D=D_i$, corresponding to boundary data $\widehat M$,
$\widehat M^*$ respectively. If
\begin{equation}
  \label{eq:3.main_hp1}
  w_1=w_2,\qquad w_{1,n}=w_{2,n},\qquad\hbox{on}\ \Gamma,
\end{equation}
\begin{equation}
  \label{eq:3.main_hp2}
  w_1^*=w_2^*,\qquad w_{1,n}^*= w_{2,n}^*,\qquad\hbox{on}\ \Gamma,
\end{equation}
then
\begin{equation}
  \label{eq:3.main_ts}
  D_1=D_2.
\end{equation}
\end{theo}

Next, let us consider the case of a plate whose boundary is composed by an accessible portion
$\Gamma$ and by an unknown inaccessible portion $I$, to be determined. More precisely, let
$\Gamma$, $I$ be two closed, nonempty
sub-arcs of the boundary $\partial \Omega$ such that
\begin{equation}
  \label{eq:3.hp-bdry1}
  \Gamma \cup I = \partial \Omega, \quad \Gamma \cap I = \{Q,R\},
\end{equation}
where $Q$, $R$ are two distinct points of $\partial \Omega$. On
the assigned couple field $\widehat M$ let us require the
following assumptions:
\begin{equation}
  \label{eq:3.emme_reg_gamma_unknown_bdry}
  \widehat M\in L^{2}(\Gamma,\mathbb{R}^2), \quad (\widehat M_n,\widehat M_{\tau,s}) \not\equiv 0,
\end{equation}
\begin{equation}
  \label{eq:3.emme_comp_gamma_unknown_bdry}
  \int_{\Gamma}\widehat M_i=0,\qquad i=1,2.
\end{equation}
Under the same assumptions made in Subsection
\ref{subsec:rigid_uniq} for the plate tensor and the domain
$\Omega$, the transversal displacement $w\in H^2(\Omega)$
satisfies the following Neumann problem
\begin{center}
\( {\displaystyle \left\{
\begin{array}{lr}
  M_{\alpha\beta,\alpha\beta}=0, &\mathrm{in}\ \Omega,
    \vspace{0.25em}\\
  M_{\alpha\beta}n_\alpha n_\beta = \widehat M_n, &\mathrm{on}\ \Gamma,
        \vspace{0.25em}\\
  M_{\alpha\beta,\beta}n_\alpha+(M_{\alpha\beta}n_\beta \tau_\alpha),_s
     =-(\widehat M_\tau),_s, &\mathrm{on}\ \Gamma,
          \vspace{0.25em}\\
  M_{\alpha\beta}n_\alpha n_\beta = 0, &\mathrm{on}\ I,
        \vspace{0.25em}\\
  M_{\alpha\beta,\beta}n_\alpha+(M_{\alpha\beta}n_\beta \tau_\alpha),_s
     =0, &\mathrm{on}\ I.
            \vspace{0.25em}\\
\end{array}
\right. } \) \vskip -9.0em
\begin{eqnarray}
& & \label{eq:3.compact_equation-bdry}\\
& & \label{eq:3.compact_bc1-Gamma}\\
& & \label{eq:3.compact_bc2-Gamma}\\
& & \label{eq:3.compact_bc1-I}\\
& & \label{eq:3.compact_bc2-I}
\end{eqnarray}

\end{center}
\noindent

Concerning the inverse problem of the determination of the unknown boundary portion $I$,
in \cite{M-R2} we have proved the following result.
\begin{theo} [Unique determination of unknown boundaries with one measurement]
\label{theo:main-unknown-bdry}
Let $\Omega_1$, $\Omega_2$ be two simply connected bounded domains
in $\mathbb{R}^2$ such that $\partial \Omega_i$, $i=1,2$, are of
class $C^{4,1}$. Let $\partial \Omega_i = I_i \cup \Gamma$,
$i=1,2$, where $I_i$ and $\Gamma$ are the inaccessible and the
accessible parts of the boundaries $\partial \Omega_i$,
respectively. Let us assume that $\Omega_1$ and $\Omega_2$ lie on
the same side of $\Gamma$ and that conditions
\eqref{eq:3.hp-bdry1} are satisfied by both pairs $\{ I_1,
\Gamma\}$ and $\{ I_2, \Gamma\}$. Let the plate tensor
$\mathbb{P}$ of class $C^{2,1}(\R^2)$ be given by
\eqref{eq:P_def}, and satisfying \eqref{eq:3.bound}
\eqref{eq:3.convex} and the dichotomy condition
\eqref{3.D(x)bound} or \eqref{3.D(x)bound 2}. Let $\widehat M \in
L^2(\Gamma, \mathbb{R}^2)$ be a boundary couple field satisfying
conditions \eqref{eq:3.emme_reg_gamma_unknown_bdry},
\eqref{eq:3.emme_comp}. Let $w_i \in H^2(\Omega_i)$ be a solution
to the Neumann problem
\eqref{eq:3.compact_equation-bdry}--\eqref{eq:3.emme_comp_gamma_unknown_bdry}
in $\Omega=\Omega_i$, $i=1,2$. If
\begin{equation}
  \label{eq:3.theorem1-hp}
    w_1=w_2, \qquad w_{1,n}=w_{2,n}, \qquad \textrm{ on } \Gamma,
\end{equation}
then
\begin{equation}
  \label{eq:3.theorem1-tesi}
    \Omega_1=\Omega_2.
\end{equation}
\end{theo}

\section{Size estimates for extreme inclusions}
\label{sec:size-estimates-extreme}

\subsection{Formulation of the problem and main results}
\label{sec:formulation-main}

Let us assume that the middle plane of the plate $\Omega$ is a
bounded domain in $\R^2$ of class $C^{1,1}$ with constants
$\rho_0$, $M_0$. In the present section we shall derive
constructive upper and lower bounds of the area of either a rigid
inclusion or a cavity in an elastic plate {}from a single boundary
measurement. These \textit{extreme} inclusions will be represented
by an open subset $D$ of $\Omega$ such that $\Omega \setminus
\overline{D}$ is connected and $D$ is compactly contained in
$\Omega$; that is, there is a number $d_0>0$ such that
\begin{equation}
  \label{eq:extreme-1.1}
    \textrm{dist} (D, \partial \Omega) \geq d_0 \rho_0.
\end{equation}
In addition, in proving the lower bound for the area of $D$, we
shall introduce the following \textit{a priori} information, which
is a way of requiring that $D$ is not ''too thin''.

\begin{definition}
  \label{def:SIFC} (Scale Invariant Fatness Condition)
Given a domain $D$ having Lipschitz boundary with constants
$r\rho_0$ and $L$, where $r>0$, we shall say that it satisfies the
Scale Invariant Fatness Condition with constant $Q>0$ if
\begin{equation}
  \label{eq:extreme-2.1}
    \textrm{diam} (D) \leq Qr \rho_0.
\end{equation}
\end{definition}

\begin{rem}
  \label{rem:extreme-pag2}
  It is evident that if $D$ satisfies Definition \ref{def:SIFC},
  then we have the trivial upper and lower estimates
\begin{equation}
  \label{eq:extreme-2.2}
    \frac{\omega_2}{ (1+  \sqrt{1+L^2})^2  }r^2\rho_0^2 \leq
    \textrm{area}(D) \leq \omega_2 Q^2 r^2 \rho_0^2,
\end{equation}
where $\omega_2$ denotes the measure of the unit disk in $\R^2$.
Since we are interested in obtaining upper and lower bounds of the
area of $D$ when $D$ is unknown, it will be necessary to consider
also the number $r$ as an unknown parameter, and all our estimates
will not depend on $r$. Conversely, the parameters $L$ and $Q$,
which are invariant under scaling, will be considered as \textit{a
priori} information on the unknown inclusion $D$.
\end{rem}

For reader's convenience and in order to introduce some useful
notation, we briefly recall the formulation of the equilibrium
problem when the plate contains either a rigid inclusion or a
cavity, and when the inclusion is absent.

Let us assume that the plate tensor $\mathbb P \in L^\infty(\R^2,
{\mathcal L} ({\mathbb{M}}^{2},
  {\mathbb{M}}^{2}))$ given by
\eqref{eq:P_def} satisfies the symmetry conditions
\eqref{eq:sym-conditions-C-components} and the strong convexity
condition \eqref{eq:3.convex}. Moreover, let $\widehat{M} \in H^{-
\frac{1}{2}}(\partial \Omega, \R^2)$ satisfy \eqref{eq:M_comp}.

If an inclusion $D$ made by rigid material is present, with
boundary $\partial D$ of class $C^{1,1}$, then the transversal
displacement $w_R$ corresponding to the assigned couple field
$\widehat{M}$ is given as the weak solution $w_R \in H^2(\Omega
\setminus \overline{D})$ of the boundary value problem
\begin{center}
\( {\displaystyle \left\{
\begin{array}{lr}
     {\rm div}({\rm div} (
      {\mathbb P}\nabla^2 w_R))=0,
      & \mathrm{in}\ \Omega \setminus \overline {D},
        \vspace{0.25em}\\
      ({\mathbb P} \nabla^2 w_R)n\cdot n=-\widehat M_n, & \mathrm{on}\ \partial \Omega,
          \vspace{0.25em}\\
      {\rm div}({\mathbb P} \nabla^2 w_R)\cdot n+(({\mathbb P} \nabla^2
      w_R)n\cdot \tau),_s
      =(\widehat M_\tau),_s, & \mathrm{on}\ \partial \Omega,
        \vspace{0.25em}\\
      w_R|_{\overline{D}} \in \mathcal{A}, &\mathrm{in}\ \overline{D},
          \vspace{0.25em}\\
        \frac{\partial w_R^e}{\partial n} = \frac{\partial w_R^i}{\partial n}, &\mathrm{on}\ \partial {D},
          \vspace{0.25em}\\
\end{array}
\right. } \) \vskip -8.9em
\begin{eqnarray}
& & \label{eq:extreme-3.1a}\\
& & \label{eq:extreme-3.1b}\\
& & \label{eq:extreme-3.1c}\\
& & \label{eq:extreme-3.1d}\\
& & \label{eq:extreme-3.1e}
\end{eqnarray}
\end{center}
coupled with the equilibrium conditions for the rigid inclusion
$D$
\begin{multline}
  \label{eq:extreme-3.2}
    \int_{\partial D} \left ( {\rm div}({\mathbb P} \nabla^2 w_R^e)\cdot n+(({\mathbb P} \nabla^2
  w_R^e)n\cdot \tau),_s \right )g - (({\mathbb P} \nabla^2 w_R^e)n\cdot n)
  g_{,n} =0, \\   \quad \hbox{for every } g\in \mathcal{A},
\end{multline}
where we recall that we have defined $w_R^e \equiv w|_{\Omega
\setminus \overline{D}}$ and $w_R^i \equiv w|_{ \overline{D}}$.

When a cavity is present, then the transversal displacement in
$\Omega \setminus \overline{D}$ is given as the weak solution $w_V
\in H^2(\Omega \setminus \overline{D})$ to the boundary value
problem
\begin{center}
\( {\displaystyle \left\{
\begin{array}{lr}
     {\rm div}({\rm div} (
      {\mathbb P}\nabla^2 w_V))=0,
      & \mathrm{in}\ \Omega \setminus \overline {D},
        \vspace{0.25em}\\
      ({\mathbb P} \nabla^2 w_V)n\cdot n=-\widehat M_n, & \mathrm{on}\ \partial \Omega,
          \vspace{0.25em}\\
      {\rm div}({\mathbb P} \nabla^2 w_V)\cdot n+(({\mathbb P} \nabla^2
      w_V)n\cdot \tau),_s
      =(\widehat M_\tau),_s, & \mathrm{on}\ \partial \Omega,
        \vspace{0.25em}\\
      ({\mathbb P} \nabla^2 w_V)n\cdot n=0, & \mathrm{on}\ \partial D,
        \vspace{0.25em}\\
    {\rm div}({\mathbb P} \nabla^2 w_V)\cdot n+(({\mathbb P} \nabla^2
        w_V)n\cdot \tau),_s=0, & \mathrm{on}\ \partial D.
          \vspace{0.25em}\\
\end{array}
\right. } \) \vskip -8.9em
\begin{eqnarray}
& & \label{eq:extreme-4.1a}\\
& & \label{eq:extreme-4.1b}\\
& & \label{eq:extreme-4.1c}\\
& & \label{eq:extreme-4.1d}\\
& & \label{eq:extreme-4.1e}
\end{eqnarray}
\end{center}
Finally, when the inclusion is absent, we shall denote by $w_0 \in
H^2(\Omega)$ the corresponding transversal displacement of the
plate which will be given as the weak solution of the Neumann
problem
\begin{center}
\( {\displaystyle \left\{
\begin{array}{lr}
     {\rm div}({\rm div} (
      {\mathbb P}\nabla^2 w_0))=0,
      & \mathrm{in}\ \Omega,
        \vspace{0.25em}\\
      ({\mathbb P} \nabla^2 w_0)n\cdot n=-\widehat M_n, & \mathrm{on}\ \partial \Omega,
          \vspace{0.25em}\\
      {\rm div}({\mathbb P} \nabla^2 w_0)\cdot n+(({\mathbb P} \nabla^2
      w_0)n\cdot \tau),_s
      =(\widehat M_\tau),_s, & \mathrm{on}\ \partial \Omega.
          \vspace{0.25em}\\
\end{array}
\right. } \) \vskip -5.9em
\begin{eqnarray}
& & \label{eq:extreme-5.2a}\\
& & \label{eq:extreme-5.2b}\\
& & \label{eq:extreme-5.2c}
\end{eqnarray}
\end{center}
We shall denote by $W_R$, $W_V$, $W_0$ the work exerted by the
couple field $\widehat{M}$ acting on $\partial \Omega$ when $D$ is
a rigid inclusion, it is a cavity, or it is absent, respectively.
By the weak formulation of the corresponding equilibrium problem
it turns out that
\begin{equation}
  \label{eq:extreme-6.1}
  W_R=-\int_{\partial\Omega}\widehat M_{\tau,s}w_R+\widehat M_n
  w_R,_n= \int_{\Omega \setminus \overline{D}} \mathbb P \nabla^2 w_R \cdot
  \nabla^2 w_R,
  \end{equation}
\begin{equation}
  \label{eq:extreme-6.2}
  W_V=-\int_{\partial\Omega}\widehat M_{\tau,s}w_V+\widehat M_n
  w_V,_n= \int_{\Omega \setminus \overline{D}} \mathbb P \nabla^2 w_V \cdot
  \nabla^2 w_V,
  \end{equation}
\begin{equation}
  \label{eq:extreme-6.3}
  W_0=-\int_{\partial\Omega}\widehat M_{\tau,s}w_0+\widehat M_n
  w_0,_n= \int_{\Omega} \mathbb P \nabla^2 w_0 \cdot
  \nabla^2 w_0.
  \end{equation}
Note that the works $W_R$, $W_V$ and $W_0$ are well defined since
they are invariant with respect to the addition of any affine
function to the displacement fields $w_R$, $w_V$ and $w_0$,
respectively. In the following, the solutions $w_V$ and $w_0$ will
be uniquely determined by imposing the normalization conditions
\begin{equation}
  \label{eq:extreme-6bis.1}
  \int_{\partial D} w_V=0, \quad \int_{\partial D} \nabla w_V=0,
  \end{equation}
\begin{equation}
  \label{eq:extreme-6bis.2}
  \int_\Omega w_0=0, \quad \int_\Omega \nabla w_0=0.
  \end{equation}
Concerning the solution $w_R$, we found convenient to normalize it
by requiring that
\begin{equation}
  \label{eq:extreme-6bis.3}
  w_R=0, \quad \mathrm{in } \ \overline{D}.
\end{equation}
For a given positive number $h_1$, we denote by $D_{h_1 \rho_0}$
the set
\begin{equation}
  \label{eq:extreme-Dh1rhozero}
    D_{h_1 \rho_0}= \{ x \in D | \ \mathrm{dist}(x, \partial D) >
    h_1 \rho_0\}.
\end{equation}
We are now in position to state our size estimates. In the case of
a rigid inclusion we have the following two theorems.

\begin{theo}
  \label{theo:extreme-rig-up}

Let $\Omega$ be a simply connected bounded domain in $\R^2$, such
that $\partial \Omega$ is of class $C^{2,1}$ with constants
$\rho_0$, $M_0$, and satisfying \eqref{eq:bound_area}. Let $D$ be
a simply connected open subset of $\Omega$ with boundary $\partial
D$ of class $C^{1,1}$, satisfying \eqref{eq:extreme-1.1}, such
that $\Omega \setminus \overline{D}$ is connected and
\begin{equation}
  \label{eq:extreme-7.1}
  \mathrm{area}(D_{h_1\rho_0}) \geq \frac{1}{2} \mathrm{area}(D),
\end{equation}
for a given positive number $h_1$. Let the plate tensor $\mathbb
P$ given by \eqref{eq:P_def} satisfy
\eqref{eq:sym-conditions-C-components}, \eqref{eq:3.convex},
\eqref{eq:3.bound_quantit} and the dichotomy condition. Let
$\widehat{M}\in L^2(\partial \Omega, \R^2)$ satisfy
\eqref{eq:reg_M}--\eqref{eq:supp_M}, with $\Gamma$ satisfying
\eqref{eq:small_enough}. The following inequality holds
\begin{equation}
  \label{eq:extreme-7.2}
  \mathrm{area}(D) \leq K\rho_0^2 \frac{W_0-W_R}{W_0},
\end{equation}
where the constant $K>0$ only depends on the quantities $M_0$,
$M_1$, $d_0$, $h_1$, $\gamma$, $\delta_1$, $M$, $\delta_0$ and
$F$.
\end{theo}

\begin{theo}
  \label{theo:extreme-rig-low}

Let $\Omega$ be a simply connected bounded domain in $\R^2$, such
that $\partial \Omega$ is of class $C^{2,1}$ with constants
$\rho_0$, $M_0$, and satisfying \eqref{eq:bound_area}. Let $D$ be
a simply connected domain satisfying \eqref{eq:extreme-1.1},
\eqref{eq:extreme-2.1}, such that $\Omega \setminus \overline{D}$
is connected and the boundary $\partial D$ is of class $C^{3,1}$
with constants $r\rho_0$, $L$. Let the plate tensor $\mathbb P$
given by \eqref{eq:P_def} satisfy
\eqref{eq:sym-conditions-C-components}, \eqref{eq:3.convex} and
\eqref{eq:3.bound_quantit}. Let $\widehat{M}\in L^2(\partial
\Omega, \R^2)$ satisfy \eqref{eq:M_comp}. The following inequality
holds
\begin{equation}
  \label{eq:extreme-8.1}
  C\rho_0^2 \Phi \left ( \frac{W_0-W_R}{W_0} \right ) \leq
  \mathrm{area}(D),
\end{equation}
where the function $\Phi$ is given by
\begin{equation}
  \label{eq:extreme-8.2}
  [0,1) \ni t \mapsto \Phi(t) = \frac{t^2}{1-t},
\end{equation}
and $C>0$ is a constant only depending on $M_0$, $M_1$, $d_0$,
$L$, $Q$, $\gamma$ and $M$.
\end{theo}

When $D$ is a cavity, the following bounds hold.

\begin{theo}
  \label{theo:extreme-cav-up}

Let $\Omega$ be a simply connected bounded domain in $\R^2$, such
that $\partial \Omega$ is of class $C^{2,1}$ with constants
$\rho_0$, $M_0$, and satisfying \eqref{eq:bound_area}. Let $D$ be
a simply connected open subset of $\Omega$ with boundary $\partial
D$ of class $C^{1,1}$, satisfying \eqref{eq:extreme-1.1}, such
that $\Omega \setminus \overline{D}$ is connected and
\begin{equation}
  \label{eq:extreme-9.1}
  \mathrm{area}(D_{h_1\rho_0}) \geq \frac{1}{2} \mathrm{area}(D),
\end{equation}
for a given positive number $h_1$. Let the plate tensor $\mathbb
P$ given by \eqref{eq:P_def} satisfy
\eqref{eq:sym-conditions-C-components},
 \eqref{eq:3.convex}, \eqref{eq:3.bound_quantit} and the dichotomy
condition. Let $\widehat{M}\in L^2(\partial \Omega, \R^2)$ satisfy
\eqref{eq:reg_M}--\eqref{eq:supp_M}, with $\Gamma$ satisfying
\eqref{eq:small_enough}. The following inequality holds
\begin{equation}
  \label{eq:extreme-9.2}
  \mathrm{area}(D) \leq K\rho_0^2 \frac{W_V-W_0}{W_0},
\end{equation}
where the constant $K>0$ only depends on the quantities $M_0$,
$M_1$, $d_0$, $h_1$, $\gamma$, $\delta_1$, $M$, $\delta_0$ and
$F$.
\end{theo}

\begin{theo}
  \label{theo:extreme-cav-low}

Let $\Omega$ be a simply connected bounded domain in $\R^2$, such
that $\partial \Omega$ is of class $C^{1,1}$ with constants
$\rho_0$, $M_0$, and satisfying \eqref{eq:bound_area}.  Let $D$ be
a simply connected domain satisfying \eqref{eq:extreme-1.1},
\eqref{eq:extreme-2.1}, such that $\Omega \setminus \overline{D}$
is connected and the boundary $\partial D$ is of class $C^{1,1}$
with constants $r\rho_0$, $L$. Let the plate tensor $\mathbb P$
given by \eqref{eq:P_def} satisfy
\eqref{eq:sym-conditions-C-components} and \eqref{eq:3.convex},
and such that $\|\mathbb P\|_{C^{2,1}(\R^2)} \leq M'$, where $M'$
is a positive parameter. Let $\widehat{M}\in L^2(\partial \Omega,
\R^2)$ satisfy \eqref{eq:M_comp}. The following inequality holds
\begin{equation}
  \label{eq:extreme-10.1}
  C\rho_0^2 \Psi \left ( \frac{W_V-W_0}{W_0} \right ) \leq
  \mathrm{area}(D),
\end{equation}
where the function $\Psi$ is given by
\begin{equation}
  \label{eq:extreme-10.2}
  [0,+\infty) \ni t \mapsto \Psi(t) = \frac{t^2}{1+t},
\end{equation}
and $C>0$ is a constant only depending on $M_0$, $M_1$, $d_0$,
$L$, $Q$, $\gamma$ and $M'$.
\end{theo}

\subsection{Proof of Theorems \ref{theo:extreme-rig-up} and \ref{theo:extreme-rig-low}}
\label{sec:proof-rigid}

The starting point of the upper and lower estimates of the area of
a rigid inclusion is the following energy estimate, in which the
works $W_0$ and $W_R$ are compared.

\begin{lem}
  \label{lem:extreme-rigid}

Let $\Omega$ be a simply connected bounded domain in $\R^2$ with
boundary $\partial \Omega$ of class $C^{1,1}$. Assume that $D$ is
a simply connected open set compactly contained in $\Omega$, with
boundary $\partial D$ of class $C^{1,1}$ and such that $\Omega
\setminus \overline{D}$ is connected. Let the plate tensor
$\mathbb P \in L^\infty( \Omega, \mathcal{L}(\mathbb M^2, \mathbb
M^2))$ given by \eqref{eq:P_def} satisfy
\eqref{eq:sym-conditions-C-components} and \eqref{eq:3.convex}.
Let $\widehat{M} \in H^{ - \frac{1}{2}}(\partial \Omega, \R^2)$
satisfy \eqref{eq:M_comp}. Let $w_R \in H^2(\Omega \setminus
\overline{D})$, $w_0 \in H^2(\Omega)$ be the solutions to problems
\eqref{eq:extreme-3.1a}--\eqref{eq:extreme-3.2} and
\eqref{eq:extreme-5.2a}--\eqref{eq:extreme-5.2c}, normalized as
above. We have
\begin{equation}
  \label{eq:extreme-11.1}
  \int_D \mathbb P \nabla^2 w_0 \cdot \nabla^2 w_0 \leq W_0 -W_R =
  \int_{\partial D} M_n(w_R)w_{0,n}+V(w_R)w_0,
\end{equation}
where we have denoted by
\begin{equation}
  \label{eq:extreme-11.2}
  M_n(w_R) = -(\mathbb P \nabla^2 w_R)n \cdot n,
\end{equation}
\begin{equation}
  \label{eq:extreme-11.3}
  V(w_R) = {\rm div}({\mathbb P} \nabla^2 w_R)\cdot n+(({\mathbb P} \nabla^2
      w_R)n\cdot \tau),_s
\end{equation}
the bending moment and the Kirchhoff shear on $\partial D$
associated to $w_R$, respectively. Here, $n$ denotes the exterior
unit normal to
 $\Omega \setminus \overline{D}$.
\end{lem}

\begin{proof} The proof is based on the weak formulation of
problems \eqref{eq:extreme-3.1a}--\eqref{eq:extreme-3.2} and
\eqref{eq:extreme-5.2a}--\eqref{eq:extreme-5.2c}, and can be
obtained by adapting the proof of the corresponding result for a
rigid inclusion in an elastic body derived in (\cite{M-R0}, Lemma
3.1).
\end{proof}

\begin{proof} [Proof of Theorem \ref{theo:extreme-rig-up}]

By \eqref{eq:extreme-11.1} and {}from the strong convexity
condition \eqref{eq:3.convex} we have
\begin{equation}
  \label{eq:extreme-13.1}
  \int_D |\nabla^2 w_0|^2 \leq \gamma^{-1}(W_0 -W_R).
\end{equation}
Estimate \eqref{eq:extreme-7.2} can be obtained {}from the
following lower bound for the elastic energy associated to $w_0$
in $D$
\begin{equation}
  \label{eq:extreme-13.2}
  \int_D |\nabla^2 w_0|^2 \geq C \frac{\mathrm{area}(D)
  }{\rho_0^2}W_0,
\end{equation}
where the constant $C>0$ only depends on $M_0$, $M_1$, $d_0$,
$h_1$, $\gamma$, $\delta_1$, $M$, $\delta_0$ and $F$. The above
estimate was derived in (\cite{M-R-V6}, Theorem 3.1) for
inclusions $D$ satisfying the \textit{fatness condition}
\eqref{eq:extreme-7.1} and its proof is based on a three sphere
inequality for solutions to the plate equation
\eqref{eq:extreme-5.2a} with anisotropic elastic coefficients
obeying to the dichotomy condition.
\end{proof}

In order to prove Theorem \ref{theo:extreme-rig-low} we need the
following Poincar\'{e} inequalities of constructive type.

For a given positive number $r>0$, we denote by $D^{r\rho_0}$ the
following set
\begin{equation}
  \label{eq:extreme-14.1}
    D^{r\rho_0}= \{ x\in \R^2 | \ 0<\mathrm{dist}(x,D) < r\rho_0
    \}.
\end{equation}
For $D$ with Lispchitz boundary and $u\in H^1(D)$ we define
\begin{equation}
  \label{eq:extreme-14.2}
    u_D= \frac{1}{|D|}\int_D u, \quad u_{\partial D} = \frac{1}{|\partial
    D|}\int_{\partial D} u.
\end{equation}
\begin{prop}
\label{prop:Poincare}

Let $D$ be a bounded domain in $\R^n$, $n \geq 2$, of Lipschitz
class with constants $r\rho_0$, $L$, satisfying condition
\eqref{eq:extreme-2.1} with constant $Q>0$. For every $u \in
H^1(D)$ we have
\begin{equation}
  \label{eq:extreme-15.1}
    \int_D |u-u_D|^2 \leq C_1 r^2 \rho_0^2 \int_D |\nabla u|^2,
\end{equation}
\begin{equation}
  \label{eq:extreme-15.2}
    \int_{\partial D} |u-u_{\partial D}|^2 \leq C_2 r \rho_0 \int_D |\nabla u|^2,
\end{equation}
where $C_1>0$, $C_2>0$ only depend on $L$ and $Q$.

If $u \in H^1(D^{r\rho_0})$, then
\begin{equation}
  \label{eq:extreme-15.3}
    \int_{\partial D} |u-u_{\partial D}|^2 \leq C_3 r \rho_0 \int_{D^{r\rho_0}} |\nabla u|^2,
\end{equation}
where $C_3 >0$ only depends on $L$ and $Q$. Moreover, if $u \in
H^1(D^{r\rho_0})$ and $u=0$ on $\partial D$, then we have
\begin{equation}
  \label{eq:extreme-15.4}
    \int_{D^{r\rho_0}} u^2 \leq C_4 r^2 \rho_0^2 \int_{D^{r\rho_0}} |\nabla u|^2,
\end{equation}
where $C_4>0$ only depends on $L$ and $Q$.
\end{prop}

\begin{proof} We refer to \cite{l:amr02} for a proof of the
inequalities \eqref{eq:extreme-15.1}--\eqref{eq:extreme-15.3} and
for a precise evaluation of the constants $C_1$, $C_2$, $C_3$ in
terms of the scale invariant bounds $L$, $Q$ regarding the
regularity and shape of $D$.

Inequality \eqref{eq:extreme-15.4} is a consequence of the
following result. Let $\Omega$ be a bounded domain in $\R^n$ with
Lipschitz boundary with constants $r\rho_0$, $L$, and such that
$\mathrm{diam}(\Omega)\leq Qr\rho_0$, $Q>0$. Let $E$ be any
measurable subset of $\Omega$ with positive Lebesgue measure
$|E|>0$. For every $v \in H^1(\Omega)$ we have
\begin{equation}
  \label{eq:extreme-16.1}
    \int_\Omega |v-v_E|^2 \leq \left (
    1 + \sqrt{ \frac{|\Omega|}{|E|}} \right )^2 Cr^2 \rho_0^2 \int_\Omega |\nabla v|^2,
\end{equation}
where the constant $C>0$ only depends on $L$ and $Q$. The above
inequality follows {}from Lemma 2.1 of \cite{l:amr08fm} and {}from
\eqref{eq:extreme-15.1} applied to the function $v$, see also
inequality (3.8) of \cite{l:amr08fm}. Let us extend the function
$u \in H^1(D^{r\rho_0})$ into the interior of $D$ by taking $u
\equiv 0$ in $D$, and let us continue to denote by $u$ this
extended function, with $u \in H^1(D \cup D^{r\rho_0})$.
Inequality \eqref{eq:extreme-15.4} follows {}from
\eqref{eq:extreme-16.1} by taking $v=u$, $\Omega = D \cup
D^{r\rho_0}$ and $E=D$.
\end{proof}

\begin{proof} [Proof of Theorem \ref{theo:extreme-rig-low}]

Let $g=a+bx+cy$ be an affine function such that the function
$\widetilde{w}_0= w_0 +g$ satisfies
\begin{equation}
  \label{eq:extreme-18.1}
    \int_{\partial D} \widetilde{w}_0=0, \quad \int_{\partial D}
    \nabla \widetilde{w}_0=0.
\end{equation}
The function $\widetilde{w}_0 \in H^2(\Omega)$ is a solution of
\eqref{eq:extreme-5.2a}--\eqref{eq:extreme-5.2c}, and by
\eqref{eq:extreme-3.2}, {}from the right-hand side of
\eqref{eq:extreme-11.1} and by applying H\"{o}lder's inequality,
we have
\begin{multline}
  \label{eq:extreme-18.2}
    W_0 - W_R
    \leq
    \left ( \int_{\partial D} |M_n(w_R)|^2 \right ) ^{
    \frac{1}{2}} \left ( \int_{\partial D} |\widetilde{w}_{0,n}|^2 \right ) ^{
    \frac{1}{2}} + \\
    +
    \left ( \int_{\partial D} |V(w_R)|^2 \right ) ^{
    \frac{1}{2}} \left ( \int_{\partial D} |\widetilde{w}_{0}|^2 \right ) ^{
    \frac{1}{2}} = I_1 + I_2.
\end{multline}
We start by estimating $I_1$. By \eqref{eq:extreme-15.2} and by
the definition of $\widetilde{w}_0$ we have
\begin{equation}
  \label{eq:extreme-18.3}
   \int_{\partial D} |\widetilde{w}_{0,n}|^2
   \leq C r\rho_0 \int_D |\nabla^2 w_0|^2 \leq C r \rho_0
   \|\nabla^2 w_0\|_{L^\infty(D)}^2 \mathrm{area}(D),
\end{equation}
where the constant $C>0$ only depends on $L$ and $Q$. By the
Sobolev embedding theorem (see, for instance, \cite{l:ad}), by
standard interior regularity estimates (see, for example, Theorem
8.3 in \cite{M-R-V1}), by Proposition \ref{prop:Poincare}, by
\eqref{eq:3.convex} and by \eqref{eq:extreme-6.3} we have
\begin{equation}
  \label{eq:extreme-19.1}
   \|\nabla^2 w_0\|_{L^\infty(D)} \leq
   \frac{C}{\rho_0^2}\|w_0\|_{H^2(\Omega)}
   \leq
   \frac{C}{\rho_0} \left ( \int_\Omega |\nabla^2 w_0|^2 \right
   )^{ \frac{1}{2}} \leq \frac{C}{\rho_0} W_0^{ \frac{1}{2}},
\end{equation}
where the constant $C>0$ only depends on $M_0$, $M_1$, $d_0$,
$\gamma$ and $M$.

The integral $\int_{\partial D} |M_n(w_R)|^2$ can be estimated by
using a trace inequality (see, for instance, \cite{Lions-Magenes})
and a $H^3$-regularity estimate up to the boundary $\partial D$
for $w_R$ (see, for example, Lemma 4.2 and Theorem 5.2 of
\cite{M-R-V5}):
\begin{equation}
  \label{eq:extreme-20.1}
   \int_{\partial D} |M_n(w_R)|^2 \leq
   \frac{C}{r^5\rho_0^5} \sum_{i=0}^3 (r\rho_0)^{2i}
   \int_{D^{ \frac{r\rho_0}{2} }} |\nabla^i w_R|^2
   \leq \frac{C}{r^5\rho_0^5} \sum_{i=0}^2 (r\rho_0)^{2i}
   \int_{D^{r\rho_0}} |\nabla^i w_R|^2,
\end{equation}
where the constant $C>0$ only depends on $M_0$, $M_1$, $L$, $Q$,
$\gamma$ and $M$. The function $w_R$ belongs to $H^2(D^{r\rho_0})$
and, by \eqref{eq:extreme-6bis.3}, $w_R=0$ and $\nabla w_R=0$ on
$\partial D$. Therefore, by applying twice
\eqref{eq:extreme-15.4}, by \eqref{eq:3.convex} and by
\eqref{eq:extreme-6.1} we have
\begin{equation}
  \label{eq:extreme-20.2}
   \sum_{i=0}^2 (r\rho_0)^{2i}
   \int_{D^{r\rho_0}} |\nabla^i w_R|^2 \leq Cr^4 \rho_0^4
   \int_{D^{r\rho_0}}|\nabla^2 w_R|^2 \leq Cr^4 \rho_0^4 W_R,
\end{equation}
where the constant $C>0$ only depends on $L$, $Q$ and $\gamma$. By
\eqref{eq:extreme-18.3}, \eqref{eq:extreme-19.1},
\eqref{eq:extreme-20.1}, \eqref{eq:extreme-20.2} we have
\begin{equation}
  \label{eq:extreme-21.1}
    I_1 \leq \frac{C}{\rho_0}W_0^{ \frac{1}{2}} W_R^{ \frac{1}{2}}
    (\mathrm{area}(D))^{ \frac{1}{2}},
\end{equation}
where the constant $C>0$ only depends on $M_0$, $M_1$, $d_0$, $L$,
$Q$, $\gamma$ and $M$.

The control of the term $I_2$ can be obtained similarly. By a
standard Poincar\'{e} inequality, by \eqref{eq:extreme-15.2} and
by \eqref{eq:extreme-19.1} we have
\begin{multline}
  \label{eq:extreme-21.2}
  \int_{\partial D} | \widetilde{w}_0|^2 \leq Cr^2\rho_0^2
  \int_{\partial D} | \widetilde{w}_{0,s}|^2 \leq Cr^3\rho_0^3
  \int_D |\nabla^2 w_0|^2 \leq \\
  \leq Cr^3\rho_0^3 \|\nabla^2
  w_0\|_{L^\infty(D)}^2 \mathrm{area}(D) \leq Cr^3 \rho_0 W_0
  \mathrm{area}(D),
\end{multline}
where the constant $C>0$ only depends on $M_0$, $M_1$, $d_0$, $L$,
$Q$, $\gamma$ and $M$. Concerning the integral $\int_{\partial D}
|V(w_R)|^2$, by using a trace inequality, a $H^4$-regularity
estimate up to the boundary $\partial D$ for $w_R$ (see, for
example, Lemma 4.3 and Theorem 5.3 in \cite{M-R-V5}) and by
\eqref{eq:extreme-20.2}, we have
\begin{multline}
  \label{eq:extreme-22.1}
   \int_{\partial D} |V(w_R)|^2 \leq
   \frac{C}{r^7\rho_0^7} \sum_{i=0}^4 (r\rho_0)^{2i}
   \int_{D^{ \frac{r\rho_0}{2} }} |\nabla^i w_R|^2
   \leq \\
   \leq \frac{C}{r^3\rho_0^3} \int_{D^{r\rho_0}} |\nabla^2 w_R|^2 \leq
   \frac{C}{r^3\rho_0^3}W_R,
\end{multline}
where the constant $C>0$ only depends on $M_0$, $M_1$, $L$, $Q$,
$\gamma$ and $M$. Then, by \eqref{eq:extreme-21.2} and
\eqref{eq:extreme-22.1} we have
\begin{equation}
  \label{eq:extreme-22.2}
    I_2 \leq \frac{C}{\rho_0}W_0^{ \frac{1}{2}} W_R^{ \frac{1}{2}}
    (\mathrm{area}(D))^{ \frac{1}{2}},
\end{equation}
where the constant $C>0$ only depends on $M_0$, $M_1$, $d_0$, $L$,
$Q$, $\gamma$ and $M$. By \eqref{eq:extreme-21.1} and
\eqref{eq:extreme-22.2}, the inequality \eqref{eq:extreme-8.1}
follows.
\end{proof}

\subsection{Proof of Theorems \ref{theo:extreme-cav-up} and \ref{theo:extreme-cav-low}}
\label{sec:proof-cavity}

As in the proof of the upper and lower estimates of the area of a
rigid inclusion, we need to compare the works $W_0$ and $W_V$. The
analogue of Lemma \ref{lem:extreme-rigid} is the following result.

\begin{lem}
  \label{lem:extreme-cavity}

Let $\Omega$ be a simply connected bounded domain in $\R^2$ with
boundary $\partial \Omega$ of class $C^{1,1}$. Assume that $D$ is
a simply connected open set compactly contained in $\Omega$, with
boundary $\partial D$ of class $C^{1,1}$ and such that $\Omega
\setminus \overline{D}$ is connected. Let the plate tensor
$\mathbb P \in L^\infty( \Omega, \mathcal{L}(\mathbb M^2, \mathbb
M^2))$ given by \eqref{eq:P_def} satisfy
\eqref{eq:sym-conditions-C-components} and \eqref{eq:3.convex}.
Let $\widehat{M} \in H^{ - \frac{1}{2}}(\partial \Omega, \R^2)$
satisfy \eqref{eq:M_comp}. Let $w_V \in H^2(\Omega \setminus
\overline{D})$, $w_0 \in H^2(\Omega)$ be the solutions to problems
\eqref{eq:extreme-4.1a}--\eqref{eq:extreme-4.1e} and
\eqref{eq:extreme-5.2a}--\eqref{eq:extreme-5.2c}, normalized as
above. We have
\begin{equation}
  \label{eq:extreme-23.1}
  \int_D \mathbb P \nabla^2 w_0 \cdot \nabla^2 w_0 \leq W_V -W_0 =
  \int_{\partial D} M_n(w_0)w_{V,n}+V(w_0)w_V \ ,
\end{equation}
where the functions $M_n(w_0)$, $V(w_0)$ are defined as in
\eqref{eq:extreme-11.2} and \eqref{eq:extreme-11.3}, and $n$
denotes the exterior unit normal to $\Omega \setminus
\overline{D}$.
\end{lem}

\begin{proof} As for Lemma \ref{lem:extreme-rigid}, the
proof is based on the weak formulation of problems
\eqref{eq:extreme-4.1a}--\eqref{eq:extreme-4.1e} and
\eqref{eq:extreme-5.2a}--\eqref{eq:extreme-5.2c}, and can be
obtained following the same guidelines of the corresponding result
for a cavity in an elastic body derived in (\cite{M-R0}, Lemma
3.5). However, here we simplify the approach presented in
\cite{M-R0} without extending the function $w_V$ in the interior
of $D$.
\end{proof}

\begin{proof} [Proof of Theorem \ref{theo:extreme-cav-up}]

The proof follows {}from the left-hand side of
\eqref{eq:extreme-23.1} by using the same arguments as in the
proof of Theorem \ref{theo:extreme-rig-up}.
\end{proof}

\begin{proof} [Proof of Theorem \ref{theo:extreme-cav-low}]

{}From the right-hand side of \eqref{eq:extreme-23.1} and by
applying H\"{o}lder's inequality, we have
\begin{multline}
  \label{eq:extreme-25.3}
    W_V - W_0
    \leq
    \left ( \int_{\partial D} |M_n(w_0)|^2 \right ) ^{
    \frac{1}{2}} \left ( \int_{\partial D} |{w}_{V,n}|^2 \right ) ^{
    \frac{1}{2}} + \\
    +
    \left ( \int_{\partial D} |V(w_0)|^2 \right ) ^{
    \frac{1}{2}} \left ( \int_{\partial D} |{w}_{V}|^2 \right ) ^{
    \frac{1}{2}} = J_1 + J_2.
\end{multline}
Let us estimate the integral $J_2$. By the Sobolev embedding
theorem (see, for instance, \cite{l:ad}), by standard interior
regularity estimates for $w_0$ (see, for example, Lemma 1 in
\cite{M-R-V2}), by \eqref{eq:extreme-15.1}, by \eqref{eq:3.convex}
and by \eqref{eq:extreme-6.3} we have
\begin{multline}
  \label{eq:extreme-26.1}
   \int_{\partial D} |V(w_0)|^2 \leq C \|\nabla^3
   w_0\|_{L^\infty(D)}^2|\partial D| \leq
   \frac{C}{\rho_0^6}\|w_0\|_{H^2(\Omega)}^2 |\partial D| \leq \\
   \leq \frac{C}{\rho_0^4} \int_\Omega |\nabla^2 w_0|^2 |\partial
   D| \leq \frac{C}{\rho_0^4}W_0|\partial D|,
\end{multline}
where the constant $C>0$ only depends on $M_0$, $M_1$, $d_0$,
$\gamma$ and $M'$. By (\cite{l:ar}, Lemma 2.8) we have
\begin{equation}
  \label{eq:extreme-26.2}
    |\partial D| \leq C \frac{\mathrm{area}(D)}{r\rho_0},
\end{equation}
where the constant $C>0$ only depends on $L$.

To control the integral $\int_{\partial D} | {w}_V|^2$ we use a
standard Poincar\'{e} inequality on $\partial D$, Proposition
\ref{prop:Poincare}, inequality \eqref{eq:extreme-15.3}, the
strong convexity condition \eqref{eq:3.convex} for $\mathbb P$ and
the definition of $W_V$, that is
\begin{multline}
  \label{eq:extreme-27.1}
   \int_{\partial D} |{w}_V|^2 \leq Cr^2\rho_0^2
   \int_{\partial D} |{w}_{V,s}|^2
   \leq Cr^3\rho_0^3 \int_{D^{r\rho_0} }|\nabla^2 w_V|^2 \leq \\
   \leq Cr^3\rho_0^3 \int_{\Omega \setminus \overline{D}} \mathbb{ P} \nabla^2 w_V \cdot
   \nabla^2 w_V = Cr^3\rho_0^3 W_V,
\end{multline}
where the constant $C>0$ only depends on $L$, $Q$ and $\gamma$.
Then, by \eqref{eq:extreme-26.1}, \eqref{eq:extreme-26.2} and
\eqref{eq:extreme-27.1}, and since, trivially, $r \leq K$, where
$K>0$ is a constant only depending on $L$ and $M_1$, we have
\begin{equation}
  \label{eq:extreme-27.2}
    J_2 \leq \frac{C}{\rho_0}W_0^{ \frac{1}{2}} W_V^{ \frac{1}{2}}
    (\mathrm{area}(D))^{ \frac{1}{2}},
\end{equation}
where the constant $C>0$ only depends on $M_0$, $M_1$, $d_0$, $L$,
$Q$, $\gamma$ and $M'$. By using similar arguments we can also
find the analogous bound for $J_1$, and the thesis follows.
\end{proof}

\end{document}